\PassOptionsToPackage{usenames, dvipsnames}{xcolor}
\documentclass[a4paper, 11pt]{preprint}
\pdfoutput=1

\usepackage{cleveref}

\usepackage{tikz}
\usetikzlibrary{decorations}
\usetikzlibrary{decorations.pathmorphing}

\tikzstyle{vertex}=[circle, fill, inner sep=2pt, outer sep=0pt]

\usepackage[normalem]{ulem}
\newcommand{\middlewave}[1]{\raisebox{0.5em}{\uwave{\hspace{#1}}}}

\tikzstyle{gray}=[shape=circle,draw=black,fill=gray!40]
\tikzstyle{black}=[shape=circle,draw=black,fill=black]
\tikzstyle{white}=[shape=circle,draw=black,fill=white]
\tikzstyle{squiggle}=[decorate,decoration=snake]
\usetikzlibrary{calc}

\usepackage{enumitem}
\usepackage{amsfonts}
\usepackage{amsmath}
\usepackage{hyperref}

\usepackage[square,numbers]{natbib} 

\usepackage{caption}
\usepackage{subcaption}

\usepackage{url}

\newcommand{\C}{\mathcal{C}}

\newcommand{\ZZ}{\mathbb{Z}}

\newcommand{\RR}{\mathbb{R}}
\newcommand{\Tc}{\mathcal{T}}

\newcommand{\vw}{\textbf{w}}
\newcommand{\vp}{\textbf{p}}
\newcommand{\e}{\mathbf{e}}
 \newcommand{\Ta}{\mathcal{T}_{\mathbf{a}}}
 \newcommand{\Tm}{\mathcal{T}_{m}}

\DeclareMathOperator{\Vol}{Vol}
\DeclareMathOperator{\vol}{vol}
\DeclareMathOperator{\Ehr}{Ehr}
\DeclareMathOperator{\codeg}{codeg}
\DeclareMathOperator{\conv}{conv}

\title{Ehrhart Theory of Cosmological Polytopes}
\author{Justus Bruckamp, Lina Goltermann, Martina Juhnke, Erik Landin and Liam Solus}
\date{\today}


\begin{document}

\begin{abstract}
The cosmological polytope of a graph $G$ was recently introduced to give a geometric approach to the computation of wavefunctions for cosmological models with associated Feynman diagram $G$. 
Basic results in the theory of positive geometries dictate that this wavefunction may be computed as a sum of rational functions associated to the facets in a triangulation of the cosmological polytope. 
The normalized volume of the polytope then provides a complexity estimate for these computations. 
In this paper, we examine the (Ehrhart) $h^\ast$-polynomial of cosmological polytopes. 
We derive recursive formulas for computing the $h^\ast$-polynomial of disjoint unions and $1$-sums of graphs. 
The degree of the $h^\ast$-polynomial for any $G$ is computed and a characterization of palindromicity is given. 
Using these observations, a tight lower bound on the $h^\ast$-polynomial for any $G$ is identified and explicit formulas for the $h^\ast$-polynomials of multitrees and multicycles are derived. 
The results generalize the existing results on normalized volumes of cosmological polytopes. 
A tight upper bound and a combinatorial formula for the $h^\ast$-polynomial of any cosmological polytope are conjectured. 
\end{abstract}

\maketitle

\section{Introduction}
\label{sec:introduction}

Let $G = (V, E)$ be an undirected graph with set of nodes $V$ and multiset of edges $E$. 
In particular, $G$ is allowed to have multiple edges as well as loops.
Let $\mathbb{R}^{V \cup E}$ denote the real-Euclidean space with standard unit vectors $\e_\alpha$ for $\alpha \in V \cup E$. 
The \emph{cosmological polytope} of $G$ is 
\[
\mathcal{C}_G = \conv(\e_i + \e_j - \e_{f}, \e_i - \e_j + \e_{f}, -\e_i + \e_j + \e_{f}~ :~ f =ij \in E) \subset \mathbb{R}^{V\cup E}. 
\]
As the cosmological polytope does not depend on isolated vertices of $G$, we only consider graphs $G$ without isolated vertices throughout the whole article. Then, the cosmological polytope $\mathcal{C}_G$ is a $(|V| + |E| - 1)$-dimensional lattice polytope
introduced by \cite{arkani2017cosmological} in the study of wavefunctions of scalar models of the universe. 
When $G$ is interpreted as a Feynman diagram representing a cosmological model, the associated wavefunction may be computed in terms of the so-called \emph{canonical form} associated to the polytope $\mathcal{C}_G$ \citep{arkani2017cosmological}. 
The canonical form $\Omega_X$ of a positive geometry $X$ is a (unique) top-dimensional differential form with logarithmic singularities only along the boundary components of $X$ \citep{lam2022invitation}. 
When the positive geometry is a $d$-dimensional convex polytope $P$, such as for $\mathcal{C}_G$, its canonical form may be expressed as 
\[
\Omega_P = \omega \frac{g}{f_1\cdots f_r}, 
\]
where $f_1,\ldots, f_r$ are the facet-defining linear forms of $P$, $\omega$ is a regular differential $d$-form on $P$ and $g$ is a polynomial, vanishing on the intersections of facets, that do not lie in $P$. The polynomial $g$ is also called the \emph{adjoint} polynomial of the dual polytope of $P$ (see e.g., \citep{Warren-adjoint}). 

A standard way to compute the canonical form $\Omega_X$ of a positive geometry is to find a subdivision of $X$ into simpler positive geometries $Y_1,\ldots, Y_k$, and apply the observation that 
\begin{equation}
\label{eqn: canform decomp}
\Omega_X = \Omega_{Y_1} + \cdots +\Omega_{Y_k}. 
\end{equation}
In physics, these subdivisions often admit physical interpretations as (possibly new) physical theories for the computation of wavefunctions. 

In the case of lattice polytopes, such as $\mathcal{C}_G$, the simplest subdivisions correspond to (unimodular) triangulations, for which the facet-defining linear forms may be recovered easily from the vertex representation of each full-dimensional simplex in the triangulation. 
Recently, \cite{jsl-cosmo} found regular, unimodular triangulations of $\mathcal{C}_G$ for any graph $G$ in terms of square-free Gr\"obner bases of the toric ideal associated to $\mathcal{C}_G$. 
They gave explicit combinatorial descriptions of the vertex representations of the full-dimensional simplices in such triangulations when $G$ is a tree or cycle. 
Such triangulations thereby yield new physical theories for the computation of the wavefunction associated to $G$ via~\eqref{eqn: canform decomp}.

The number of full-dimensional simplices in such a triangulation correspond to the number of terms to be computed in the sum~\eqref{eqn: canform decomp}. 
Since the triangulations found by \cite{jsl-cosmo} are unimodular, the normalized volume of $\mathcal{C}_G$ provides complexity estimates on the computation of the desired wavefunction via the triangulation. 
This motivates the study of the normalized volume of the cosmological polytope and its refinements via combinatorial generating polynomials, such as the $h^\ast$-polynomial.

In this paper, we study the $h^\ast$-polynomial of the cosmological polytope $\mathcal{C}_G$. 
We prove that the $h^\ast$-polynomial of a cosmological polytope is multiplicative with respect to disconnected components (\Cref{prop:union}) as well as the $1$-sum graph operation (\Cref{thm: 1-sum}). 
We further provide tight lower bounds on the coefficients of the $h^\ast$-polynomial of $\mathcal{C}_G$ for any $G$ (\Cref{thm: lower bound}), which naturally yield lower bounds on the complexity of the computation of the canonical form via any unimodular triangulation. 
We also derive properties of the $h^\ast$-polynomial of $\mathcal{C}_G$ that are of general combinatorial interest, including its degree (\Cref{thm: degree}) and a characterization of palindromicity (\Cref{thm: gorenstein}). 
Finally, in \Cref{sec: examples}, we explicitly compute the $h^\ast$-polynomials of cosmological polytopes of multitrees (\Cref{thm:Ehrhart multitree}) and multicycles (\Cref{thm: h-star of cycle with multiedges}) using the triangulations of \cite{jsl-cosmo}. 
By evaluating the resulting polynomials at $1$, we obtain normalized volume formulas generalizing the results of \cite{kuhne2022faces} for simple trees and \cite{jsl-cosmo} for simple cycles. 
Some resulting conjectures are posed, including a proposed tight upper bound on the $h^\ast$-polynomial of a cosmological polytope (\Cref{conj: upper bound}) and a general combinatorial formula for the $h^\ast$-polynomial of the cosmological polytope for any graph $G$ (\Cref{conj: h*-formula}).

\section{The toric ideal of a cosmological polytope and triangulations}
In this section, we describe the Gr\"obner basis for the toric ideal of a cosmological polytope from \cite{jsl-cosmo} and the corresponding triangulation. 
In \Cref{thm: reduced gens}, we prove a result to be used later in the paper; namely, that a strict subset of this Gr\"obner basis is also a generating set for the ideal. 

\subsection{Ehrhart theory and half-open decompositions}
\label{sec:ehrhart}
We start by briefly reviewing the necessary Ehrhart theory (see \cite{ccd} and the original paper by \cite{Ehrhart} for more details). 

Given a  $d$-dimensional convex lattice polytope $P\subseteq\RR^n$ and a positive integer $t$, 
we let $tP = \{t\vp ~:~ \vp \in P\}$ denote the $t$-th \emph{dilate} of $P$. 
The \emph{Ehrhart series} of $P$ is 
\[
\Ehr(P; z) = 1 + \sum_{t > 0}|tP\cap\ZZ^n|z^t. 
\]
The Ehrhart series of $P$ may be expressed as a rational function
\[
\Ehr(P;z) = \frac{h^\ast(P;z)}{(1 - z)^{d+1}}=\frac{h^\ast_d(P)z^d+\ldots+h_1^\ast(P)z+h_0^\ast(P)}{(1-z)^{d+1}}, 
\]
in which the numerator $h^\ast(P;z)$ is a polynomial of degree at most $d$ with only nonnegative integer coefficients \cite{Stanley-Ehrhart} and constant term equal to $1$. The volume of $P$, denoted by $\vol(P)$, is equal to $h^\ast(P; 1)/d!$ and the value $\Vol(P) = d!\vol(P)$ is called the \emph{normalized volume} of $P$. 
If $P$ admits a \emph{unimodular} triangulation, i.e., a triangulation whose maximal simplices have normalized volume equal to $1$, then its normalized volume equals the number of maximal simplices in the triangulation. 
Moreover, such triangulations can be used to compute the $h^\ast$-polynomial via so-called \emph{half-open decompositions}, which we will now briefly describe. 

We say that a point $q \in \mathrm{aff}(P)$ is in \emph{general position relative to $P$}, if $q$ is not contained in the affine hull $\mathrm{aff}(F)$ of any facet $F \subseteq P$. 
Assuming $q$ is in general position relative to $P$,  a facet $F \subseteq P$ is called \emph{visible} from $q$, if $[p,q] \cap P=\{p\}$ for all $p \in F$, where $[p,q]$ denotes the line segment from $p$ to $q$. 
The facet $F \subseteq P$ gives rise to the hyperplane $\mathcal{H}=\mathrm{aff}(F)$ in  $\mathrm{aff}(P)$ that defines two closed halfspaces, and $F$ is visible from $q$ if and only if $q$ and $P$ do not lie in the same halfspace. 
By $\mathrm{vis}_q(P)$ we denote the set of facets of $P$ that are visible from $q$, i.e., 
\[
\mathrm{vis}_q(P)=\{F \subseteq P \mid F \text{ is a facet of } P \text{ and visible from } q\}.
\]
Given a polytope $P \subseteq \RR^n$, a triangulation $\mathcal{T}$ of $P$ with maximal simplices $\Delta_1,\ldots,\Delta_m$, and a point $q$ in general position relative to $P$ and to $\Delta_1,\ldots,\Delta_m$, let $\mathbb{H}_q\Delta_i$ be all points of $\Delta_i$ that are not contained in a from $q$ visible facet of $\Delta_i$ for $i=1,\ldots,m$. Such a set $\mathbb{H}_q\Delta_i$ is also called a \emph{half-open simplex}, and $\mathbb{H}_q\Delta_1,\ldots,\mathbb{H}_q\Delta_m$ yield a so-called \emph{half-open decomposition} of $P$, since $P=\mathbb{H}_q\Delta_1 \dot{\cup}\ldots\dot{\cup}\mathbb{H}_q\Delta_m$ is the disjoint union of $\mathbb{H}_q\Delta_1,\ldots,\mathbb{H}_q\Delta_m$ \cite[Lemma 5.3.4]{crt}. This implies
\begin{align*}
    \vert t P \cap \ZZ^n\vert=\vert t \mathbb{H}_q\Delta_1 \cap \ZZ^n\vert+\ldots+\vert t \mathbb{H}_q\Delta_m \cap \ZZ^n\vert,
\end{align*}
and with that one can show the following statement, which is Theorem 5.5.3 in \cite{crt} and relates half-open decompositions with $h^\ast$-polynomials:

\begin{theorem}
    \label{thm: visibility formula for h star}
    Let $P \subseteq \RR^n$ be a lattice polytope that admits a unimodular triangulation $\mathcal{T}$. Let $\Delta_1,\ldots,\Delta_m$ denote the maximal simplices of $\mathcal{T}$, and let $q \in P$ be in general position relative to~$P$, and to  all $\Delta_i$, $i \in [m]$. Then
    \begin{align*}
        h^\ast_i(P)=\vert \{j \in [m] \colon \vert \mathrm{vis}_q(\Delta_j)\vert=i\}\vert.
    \end{align*}
\end{theorem}
In other words, $h^\ast_i(P)$ counts the number of maximal simplices in $\mathcal{T}$, that have exactly $i$ visible facets from $q$. 
 
\subsection{Toric ideals and Gr\"obner bases}
In the following, we provide the necessary background concerning toric ideals, restricted to the case that the underlying polytope is a cosmological polytope. 
We refer the reader to \cite{sturmfels1996grobner} for a general discussion of Gr\"obner bases and toric ideals. 
For our purposes it is sufficient to understand how one extracts a triangulation from a Gr\"obner basis, which is described in the following. 

We start with the definition of the toric ideal of a cosmological polytope.  
Given a graph $G = (V, E)$, we let $R_G$ be a polynomial ring over a field $K$ whose variables correspond to the lattice points of $\C_G$. More precisely, we have 
\begin{itemize}
    \item a variable $z_k$ for every $k \in V\cup E$, 
    \item variables $y_{ijf}$ and $y_{jif}$ for every $f = ij\in E$, and 
    \item a variable $t_{f}$ for every $f \in E$.
\end{itemize}
These three types of variables are collectively referred to as \emph{$z$-variables, $y$-variables} and \emph{$t$-variables}, respectively. 
The $z$-variables correspond to the unit vectors $\e_k$ which are contained in $\C_G$ for every $k\in (V\cup E)$, whereas the variables $y_{ijf}$, $y_{jif}$ and $t_{f}$, respectively, correspond to the vertices $\e_i-\e_j+\e_f$, $-\e_i+\e_j+\e_f$ and $\e_i+\e_j-\e_f$, respectively. 
Note that there are no other lattice points in $\C_G$.
Similarly, let $K[\C_G] = K[\vw^\vp~ : ~\vp\in (\C_G\cap \ZZ^{V\cup E})]$ be the \emph{toric ring} of $\C_G$; i.e., the ring generated by all (Laurent) monomials $\vw^\vp = \prod_{v\in V}w_v^{p_v}\prod_{f\in E}w_f^{p_f}$, where $\vp$ is a lattice point in $\C_G$.
The \emph{toric ideal} of $\C_G$, denoted $I_{\C_G}$, is defined as the kernel of the $K$-algebra homomorphism
\begin{equation*}
    \begin{split}
        \varphi_G:\, R_G &\longrightarrow K[\C_G]; \\
        z_k &\longmapsto w_k,\\
        y_{ijf} &\longmapsto w_iw_j^{-1}w_f,\\
        y_{jif} &\longmapsto w_i^{-1}w_jw_f,\\
        t_f &\longmapsto w_iw_jw_f^{-1}.
    \end{split}
\end{equation*}
Since this homomorphism is surjective, we obtain a presentation of $K[\C_G]$, and we have $K[\C_G]\cong R_G/I_{\C_G}$. 
The \emph{Hilbert series} of $K[\C_G]$ is defined as 
\[
\mathrm{Hilb}(K[\C_G]; z)=\sum_{n\geq 0}\dim_K(R_G/I_{\C_G})_nz^n,
\]
where the grading on $R_G/I_{\C_G}$ is defined via 
\[
\deg(z_k) = \deg(y_{ijf}) = \deg(y_{jif}) = \deg(t_f) = 1 
\]
and $(R_G/I_{\C_G})_n$ denotes the homogeneous component of $R_G/I_{\C_G}$ of degree $n$. 
Since $\C_G$ admits a unimodular triangulation due to \cite{jsl-cosmo}, $\C_G$ has the integer decomposition property (see, e.g. \cite[Proposition 2.60]{bruns2009gubeladze}) meaning that every lattice point $\vp \in t\C_G \cap \ZZ^{V \cup E}$ can be written as
$\vp=\vp_1+\cdots+\vp_t$, with $\vp_i \in \C_G \cap \ZZ^{V \cup E}$ for $i \in [t]$. 
With this, it follows from well-known results for general lattice polytopes that
\begin{equation}\label{eq:Hilb}
\mathrm{Hilb}(K[\C_G];z)=\Ehr(\C_G;z)=\frac{h^\ast(\C_G;z)}{(1-z)^{\dim\C_G+1}}.
\end{equation}
Hence, the Hilbert series of $K[\C_G]$ can be used to compute the $h^\ast$-polynomial of $\C_G$.

To access the $h^\ast$-polynomial of $\C_G$ from this algebraic perspective, it will be useful to understand generating sets for the toric ideal $I_{\C_G}$. 
Such a set is provided by a Gr\"obner basis of $I_{\C_G}$, and we now describe the Gr\"obner bases found in \cite{jsl-cosmo} for any cosmological polytope $\C_G$. 

Given $u,v \in V$ and a path $P$ in $G$ between $u$ and $v$ of length at least $2$, a partition $P_1\cup P_2$ of the edge set $E(P)$ of $P$ such that $P_1$ and $P_2$ contains the unique edge of $P$ incident to $u$ and $v$, respectively, gives rise to a so-called \emph{zig-zag pair} $(E_1,E_2)$ of $G$, where $E_1$ and $E_2$ is the set of edges from $P_1$ and $P_2$ directed  towards $v$ and $u$, respectively. 
Similarly, given a cycle $C$ of $G$ and a partition $C_1\cup C_2$ of $E(C)$ into two blocks (with one possibly empty), we let $E_1$ and $E_2$ be the set of edges in $C_1$ and $C_2$ oriented clockwise and counter-clockwise, respectively (where we have fixed an orientation beforehand). 
The pair $(E_1,E_2)$ is then called a \emph{cyclic pair}. In the following, given an edge $ij\in E$, we write $i\to j$ and $j\to i$ for the corresponding directed edges, directed from $i$ to $j$ and $j$ to $i$, respectively.

\begin{definition}
 Given a zig-zag pair $(E_1,E_2)$ its  \emph{zig-zag binomial} is
   \[
        b_{E_1,E_2} =
       \displaystyle z_{v}\prod_{f=i\to j\in E_1}y_{ijf}\prod_{f=i \to j\in E_2} z_f - z_{u}\prod_{f=i \to j\in E_2}y_{ijf}\prod_{f=i \to j\in E_1} z_f.
    \]   
   Given a cyclic pair $(E_1,E_2)$ its \emph{cyclic binomial} is
    \[            
    b_{E_1,E_2} = \displaystyle\prod_{f=i\to j\in E_1}y_{ijf}\prod_{f=i\to j\in E_2} z_f - \prod_{f=i \to j\in E_2}y_{ijf}\prod_{f= i\to j\in E_1} z_f.
    \]
    When either $E_1$ or $E_2$ is empty, the cyclic binomial $b_{E_1,E_2}$ is referred to as a \emph{cycle binomial}. 
\end{definition}

It is not hard to see that all binomials from the previous definition belong to the toric ideal $I_{\C_G}$. Moreover, the same can be easily verified for the  following set of binomials, referred to as \emph{fundamental binomials}:
\begin{equation}
\label{eqn: fundamental binomials}
\begin{split}
F_G = \{ &\underline{y_{ijf}y_{jif}}-z_f^2,~ \underline{y_{ijf}t_{f}}-z_i^2,~  \underline{y_{jif}t_{f}}-z_j^2,\\
		 &\underline{y_{ijf}z_j} - z_iz_f,~ \underline{y_{jif}z_i} - z_jz_f,~ \underline{t_{f}z_f} - z_iz_j~ :~ f=ij\in E\}.
\end{split}
\end{equation}

We note that cycle binomials have the property that one term consists only of $y$-variables and the other consists only of $z$-variables. With this in mind, \cite{jsl-cosmo} defined a special family of term orders on the ring $R_G$. 
A \emph{term order} $\preceq$ is a total order on the monomials of $R_G$ that is compatible with multiplication and  with the additional property that $1$ is the smallest monomial. Given a polynomial $p$ in $R_G$, its \emph{leading term} $\mathrm{lt}_{\preceq}(p)$ with respect to $\preceq$ is the monomial (appearing with non-zero coefficient) that is maximal with respect to $\preceq$.

\begin{definition}
    \label{def: good term order}
    A term order on $R_G$  is \emph{good} if the leading terms of the binomials in $F_G$ are the underlined terms in~\eqref{eqn: fundamental binomials}, and the leading terms of all cycle binomials are the terms containing only $y$-variables. 
\end{definition}

The following theorem, derived in \citep[Theorem~2.9]{jsl-cosmo}, gives a Gr\"obner basis for $I_{\C_G}$.

\begin{theorem}\citep[Theorem 2.9]{jsl-cosmo} \label{thm:GB}
The set 
\begin{equation}
\label{eqn: GB}
		B_G=F_G\cup \{b_{E_1,E_2} ~ : ~ (E_1,E_2) \text{ is a zig-zag pair or a cyclic pair of $G$}\}
\end{equation}
 is a Gr\"obner basis for the toric ideal $I_{\C_G}$ of the cosmological polytope of $G$ with respect to any good term order. 
\end{theorem}

We remark that in \cite{jsl-cosmo} the authors assumed that $G$ is a connected graph, but the statement is also true for disconnected graphs since the proof also works in that case.

Being a Gr\"obner basis, in particular implies that $B_G$ is a generating set for $I_{\C_G}$.
The next result shows, that $I_{\C_G}$ has  a significantly smaller generating set which will be crucial for the proof of \Cref{thm: 1-sum}.

\begin{theorem}
\label{thm: reduced gens}
Let $G=(V,E)$ be a graph and $J_G$ be the ideal generated by the fundamental binomials of $G$. Then, $b_{E_1,E_2}\in J_G$ for any zig-zag pair $(E_1,E_2)$ of $G$. In particular, $I_{\C_G}$ is generated by the following set
\[
\widetilde{B}_G=F_G \cup \{b_{E_1,E_2} ~ : ~ (E_1,E_2) \text{ is a cyclic pair of $G$}\}.
\]
\end{theorem}

\begin{proof}
 To show the claim, we prove a stronger statement: Given a path $P$ of $G$ from $u$ to $v$ of length at least $2$, a partition $(P_1,P_2)$ of the edges of $P$, such that $P_1$ contains the unique edge of $P$ incident to $u$, we let $E_1$ and $E_2$ be the edges of $P_1$ and $P_2$ directed towards $v$ and $u$, respectively. The pair $(E_1,E_2)$ is called an \emph{almost zig-zag pair}. Note that $(E_1,E_2)$ is a zig-zag pair if and only if the edge incident to $v$ belongs to $E_2$. In particular, we allow that $P_2=E_2=\emptyset$. An almost zig-zag pair gives rise to a zig-zag binomial in exactly the same way as a zig-zag pair does. Moreover, we will say that a zig-zag pair has \emph{length} $k$ if the underlying path has, i.e., $|E_1\cup E_2|=k$. We show by induction on the length that $b_{(E_1,E_2)}\in J_G$ for every \emph{almost} zig-zag pair $(E_1,E_2)$. 

For the base of induction, assume that $(E_1,E_2)$ is an almost zig-zag pair of length $2$. Let $P=uvw$ be a path of $G$ of length $2$, such that $e=uv\in E$ and $f=vw\in E$ and let $(E_1,E_2)$ be an almost zig-zag pair for $P$. In the following, we write $y_{uv}$ and $y_{vu}$ for the variables $y_{uve}$ and $y_{vue}$, respectively, and similar for other edges, since no confusion can arise due to possible multiple edges.

If $(E_1,E_2)$ is not just an almost zig-zag pair, but a zig-zag pair in the original sense, then we have $E_1=\{u\to v\}$ and $E_2=\{w \to v\}$. The corresponding zig-zag binomial is then
\[
b_{E_1,E_2}=z_wy_{uv}z_f-z_ey_{wv}z_u.
\]
We can rewrite this binomial as follows
\[
b_{E_1,E_2}=y_{uv}(z_wz_f-y_{wv}z_v)-y_{wv}(z_uz_e-y_{uv}z_v).
\]
Since $z_wz_f-y_{wv}z_v$ and $z_uz_e-y_{uv}z_v$ are fundamental binomials of the edge $f$ and $e$, respectively, we conclude that $b_{E_1,E_2}\in J_G$. 

If $(E_1,E_2)$ is an almost zig-zag pair but not a zig-zag pair, then we have $E_1=\{u\to v, v\to w\}$ and $E_2=\emptyset$. The corresponding zig-zag binomial can then be written as
\begin{align*}
b_{E_1,E_2}=&z_wy_{uv}y_{vw}-z_uz_ez_f\\
=&y_{uv}(z_wy_{vw}-z_vz_f)+z_f(z_vy_{uv}-z_uz_e).
\end{align*}
Since $z_wy_{vw}-z_vz_f$ and $z_vy_{uv}-z_uz_e$ are again fundamental binomials for the edge $f$ and $e$, respectively, if follows that $b_{E_1,E_2}\in J_G$.

For the induction step, let $k>2$ and consider a path $P$ of length $k$ in $G$ with $E(P)=\{i_1i_2,i_2i_3,\ldots,i_ki_{k+1}\}$. Let $P_1\cup P_2$ be a partition of $E(P)$ with $i_1i_2\in P_1$ and let $(E_1,E_2)$ be the corresponding almost zig-zag pair. The corresponding zig-zag binomial is then
\begin{equation}\label{eq:reduce}
 b_{E_1,E_2} =
       \displaystyle z_{i_{k+1}}\prod_{e=i\to j\in E_1}y_{ij}\prod_{e=i \to j\in E_2} z_e - z_{i_1}\prod_{e=i \to j\in E_2}y_{ij}\prod_{e=i \to j\in E_1} z_e.
\end{equation}
Let $f=i_ki_{k+1}$ be the last edge of $P$. To simplify notation, in the following, we will denote by $\overrightarrow{f}$ and $\overleftarrow{f}$ the edge $f$ directed towards $i_{k+1}$ and $i_k$, respectively. We distinguisth two cases.

{\sf Case 1:}  $f\in P_2$, i.e., $(E_1,E_2)$ is a zig-zag pair.
Since $f\in P_2$, the first monomial on the right-hand side of \eqref{eq:reduce} is divisible by $z_fz_{i_{k+1}}$. Using the fundamental binomial $z_{i_k}y_{i_{k+1}i_k}-z_fz_{i_{k+1}}$ of $f$, we can rewrite $b_{E_1,E_2}$ as follows
\begin{align*}
b_{E_1,E_2}= &\displaystyle (z_fz_{i_{k+1}}-z_{i_k}y_{i_{k+1}i_k}) \prod_{e=i\to j\in E_1}y_{ij}\prod_{e=i \to j\in E_2\setminus\{\overleftarrow{f}\}} z_e \\
&\displaystyle- z_{i_1}\prod_{e=i \to j\in E_2}y_{ij}\prod_{e=i \to j\in E_1} z_e +z_{i_k}y_{i_{k+1}i_k}\prod_{e=i\to j\in E_1}y_{ij}\prod_{e=i \to j\in E_2\setminus\{\overleftarrow{f}\}} z_e\\
=&\displaystyle (z_fz_{i_{k+1}}-z_{i_k}y_{i_{k+1}i_k}) \prod_{e=i\to j\in E_1}y_{ij}\prod_{e=i \to j\in E_2\setminus\{\overleftarrow{f}\}} z_e \\
&\displaystyle+ y_{i_{k+1}i_k}\left(-z_{i_1}\prod_{e=i \to j\in E_2\setminus\{\overleftarrow{f}\}}y_{ij}\prod_{e=i \to j\in E_1} z_e +z_{i_k}\prod_{e=i\to j\in E_1}y_{ij}\prod_{e=i \to j\in E_2\setminus\{\overleftarrow{f}\}} z_e\right),
\end{align*}
where for the second equality we have used that $\overleftarrow{f}\in E_2$. 
Since $f$ is the unique edge of $P$ incident to $i_{k+1}$ and since $(E_1,E_2)$ is an almost zig-zag pair, so is $(E_1,E_2\setminus \{\overleftarrow{f}\})$. In particular, the binomial in the last parentheses above equals $b_{E_1,E_2\setminus\{\overleftarrow{f}\}}$, which, by induction, belongs to $J_G$. The claim then follows since also $(z_fz_{i_{k+1}}-z_{i_k}y_{i_{k+1}i_k})\in J_G$. 

{\sf Case 2:} $f\notin P_2$ and hence, $f\in P_1$. This implies that the first monomial on the right-hand side of \eqref{eq:reduce} is divisible by $z_{i_{k+1}}y_{i_ki_{k+1}}$ and, similar to the first case, we can rewrite $b_{E_1,E_2}$ as follow
\begin{align*}
b_{E_1,E_2}= &\displaystyle (z_{i_{k+1}}y_{i_ki_{k+1}}-z_{i_k}z_f) \prod_{e=i\to j\in E_1\setminus \{\overrightarrow{f}\}}y_{ij}\prod_{e=i \to j\in E_2} z_e \\
&\displaystyle- z_{i_1}\prod_{e=i \to j\in E_2}y_{ij}\prod_{e=i \to j\in E_1} z_e +z_{i_k}z_f\prod_{e=i\to j\in E_1\setminus \{\overrightarrow{f}\}}y_{ij}\prod_{e=i \to j\in E_2} z_e \\
=&\displaystyle (z_{i_{k+1}}y_{i_ki_{k+1}}-z_{i_k}z_f) \prod_{e=i\to j\in E_1\setminus \{\overrightarrow{f}\}}y_{ij}\prod_{e=i \to j\in E_2} z_e \\
&\displaystyle+z_f\left(z_{i_k}\prod_{e=i\to j\in E_1\setminus \{\overrightarrow{f}\}}y_{ij}\prod_{e=i \to j\in E_2} z_e -z_{i_1}\prod_{e=i \to j\in E_2}y_{ij}\prod_{e=i \to j\in E_1\setminus\{\overrightarrow{f}\}} z_e\right)\\
\end{align*}
Similar to Case 1, it follows that $(E_1\setminus\{\overrightarrow{f}\},E_2)$ is an almost zig-zag pair and the binomial in the last parentheses above equals $b_{E_1\setminus\{\overrightarrow{f}\},E_2}$, which by induction belongs to $J_G$. As $z_{i_{k+1}}y_{i_ki_{k+1}}-z_{i_k}z_f$ is a fundamental binomial, we infer $b_{E_1,E_2}\in J_G$.

The \emph{in particular}-statement is now an immediate consequence.
\end{proof}

\subsection{Unimodular triangulations of cosmological polytopes}\label{sec:unimodular}
The goal of this subsection is to describe how to deduce a unimodular triangulation from \Cref{thm:GB}. It follows from a well-known result of \cite{sturmfels1996grobner} that, for a fixed graph $G$, every Gr\"obner basis of the toric ideal $I_{\C_G}$ gives rise to a regular triangulation of the cosmological polytope $\C_G$. 
Moreover, since the Gr\"obner basis from \Cref{thm:GB} has only squarefree leading terms the corresponding triangulation is unimodular. 
By the definition of $R_G$, every squarefree monomial $m$ in $R_G$ corresponds to a unique set of lattice points $S_m$ in $\C_G$. 
We will say that a given set $S$ of lattice points in $\C_G$ \emph{contains a monomial $m\in R_G$} if $S_m\subseteq S$. The precise description of the triangulation corresponding to $B_G$ for a fixed term order $\preceq$ is as follows:

\begin{corollary}\label{cor: unimodular triangulation}
    Let $G=(V,E)$ be a graph and let $\preceq$ be a good term order on $R_G$. The set
    \[
    \mathcal{T}_G=\{\conv(S)~:~ S\subseteq \C_G\cap\ZZ^{V\cup E},\; |S|=|V|+|E|,\; S \text{ does not contain $\mathrm{lt}_{\preceq}(b)$ for }b\in B_G\}
    \]
  provides a  regular unimodular triangulation of $\C_G$.
\end{corollary}

This result was shown in \cite{jsl-cosmo}.
We include it here since we will use it to construct the corresponding unimodular triangulations of the cosmological polytopes of multitrees and multicycles for specifically chosen term orders.

Following \cite{jsl-cosmo} we will represent facets of the triangulation $\mathcal{T}_G$ not as their set of lattice points but via a certain decorated subgraph of $G$. To do so we represent lattice points, or equivalently, variables of $R_G$ by four different types of edges and two different kinds of nodes. More precisely, given an edge $f=ij\in E$:
\begin{itemize}
    \item[$\bullet$] the variable $z_i$ (i.e., the lattice point $\e_i$) is represented by a \emph{white} node $\circ$.  The vertex $i$ is instead represented by a \emph{black} node $\bullet$ if $z_i$ is not present. 
    \item[$\bullet$] the variable $z_{f}$ (i.e., the lattice point $\e_f$) is represented by the edge type $-$,
    \item[$\bullet$] the variable $t_{e}$ (i.e., the lattice point $\e_i+\e_j-\e_f$) is represented by the edge type $\middlewave{0.5cm}$, called \emph{squiggly} edge in the following,
    \item[$\bullet$] the variable $y_{ijf}$ (i.e., the lattice point $\e_i-\e_j+\e_f$) is represented by the edge type $\rightarrow$ pointing from $i$ to $j$, and 
    \item[$\bullet$] the variable $y_{jif}$ (i.e., the lattice point $-\e_i+\e_j+\e_f$) is represented by the edge type $\leftarrow$ pointing from $j$ to $i$. 
\end{itemize}
Given a set $S$ of lattice points of $\C_G$, we let $G_S$ be the graph drawn according to the above interpretations of the points in $S$. This interpretation will be used in Sections \ref{subsec: cycles}  and \ref{subsec: tree polynomial}.

\section{General results on \texorpdfstring{$h^\ast$}{h*}-polynomials of cosmological polytopes}
In this section, we provide general results for $h^\ast$-polynomials of cosmological polytopes. In particular, we study the behavior of the $h^\ast$-polynomial under disjoint unions and $1$-sums of graphs. Moreover, we explicitly compute its degree and characterize graphs whose cosmological polytope is Gorenstein, i.e., which have palindromic $h^\ast$-polynomial.

\subsection{Graph operations for computing \texorpdfstring{$h^\ast$}{h*}-polynomials}
\label{subsec: graph operations}

Let $G=(V,E)$ and $H=(W,F)$ be graphs on disjoint vertex sets. 
Given vertices $v\in V$ and $w\in W$, the \emph{$1$-sum} $G\oplus_{v,w}H$ of $G$ and $H$ with respect to $v$ and $w$ is the graph obtained by taking the union of $G$ and $H$ and identifying the vertices $v$ and $w$. More formally, $G\oplus_{v,w} H$ is the graph on vertex set $(V\cup W)\setminus\{v,w\}\cup\{x_{vw}\}$ and edge set
\[
(E\cup F)\setminus (\{uv~:~uv\in E\}\cup \{uw~:~uw\in F\})\cup \{ux_{vw}~:~uv\in E \mbox{ or } uw\in F\}.
\]
Note that if $G$ and $H$ have multiple edges containing $v$ and $w$, respectively, then also $G\oplus_{v,w}H$ has multiple edges containing $x_{vw}$. 
In particular, the edge set of a graph may be a multiset.

When the identified vertices $v,w$ are clear from the context, we will write $G\oplus H$ instead of $G\oplus_{v,w} H$. 
Using the Hilbert series expression for $h^\ast(\C_G;z)$ in~\Cref{eq:Hilb}, we obtain the following.

\begin{proposition}\label{prop:union}
Let $G$ be a graph with connected components $G_1,\ldots,G_k$. Then,
\[
h^{\ast}(\C_{G};z)=\prod_{i=1}^k h^\ast(\C_{G_i};z).
\]
\end{proposition}

\begin{proof}
    It suffices to show the statement for $k=2$. Let $G_1=(V_1,E_1)$ and $G_2=(V_2,E_2)$ be the connected components of $G=(V,E)$. 
    Since the only lattice points contained in $\C_G$ are its vertices and the unit vectors corresponding to the vertices and edges of $G$, we have 
    \[
    \C_G\cap \ZZ^{V\cup E}=(\C_{G_1}\cap \ZZ^{V_1\cup E_1})\times \{0\}^{V_2\cup E_2}\cup \{0\}^{V_1\cup E_1}\times (\C_{G_2}\cap \ZZ^{V_2\cup E_2}),
    \]
    where we order the coordinates in such a way that the coordinates corresponding to $V_1\cup E_1$ come first. Note that the above union is disjoint. It follows that 
    $R_{G}\cong R_{{G_1}}\otimes_K R_{{G_2}}$ and $I_{G}=I_{{G_1}}+I_{{G_2}}$, where we abuse notation and consider $I_{{G_1}}$ and $I_{{G_2}}$ as ideals in $R_{G}$. 
    Hence, we obtain
    \begin{align*}
    K[\C_G]\cong R_{G}/I_{G}
    &\cong (R_{{G_1}}\otimes_K R_{{G_2}})/(I_{{G_1}}+I_{{G_2}})\\
    &\cong  R_{{G_1}}/I_{{G_1}}\otimes_K  R_{{G_2}}/I_{{G_2}}\\
    &\cong K[\C_{G_1}]\otimes_K K[\C_{G_2}]
    \end{align*}
    which implies that 
    \[
    \mathrm{Hilb}(K[\C_G];z)=\mathrm{Hilb}(K[\C_{G_1}];z)\cdot \mathrm{Hilb}(K[\C_{G_2}];z).
    \]
    As $\dim \C_G=\dim\C_{G_1}+\dim\C_{G_2}+1$, the claim follows from \eqref{eq:Hilb}.
\end{proof}

With the help of \Cref{thm: reduced gens}, we may then prove the following general result.

\begin{theorem}
    \label{thm: 1-sum}
    Let $G=G_1\oplus G_2$ be a $1$-sum of graphs $G_1$ and $G_2$. Then
    \[
    h^\ast(\C_G;z)=h^\ast(\C_{G_1};z)\cdot h^\ast(\C_{G_2};z).
    \]
\end{theorem} 

\begin{proof}
Let $v$ and $w$ be the vertices of $G_1$ and $G_2$, respectively, that have been identified by taking a $1$-sum. Further, we use $u$ to denote the new vertex in $G$, that is obtained by the identification of $v$ and $w$. Let $H$ be the disjoint union of $G_1$ and $G_2$. Let $ \Phi: V(H)\to V(G)$ be the natural surjective map that sends $v$ and $w$ to $u$ and that is the identity on all the other vertices. 
        $\Phi$ induces a surjective (even bijective) map from $E(H)$ to $E(G)$ by sending $ij\in E(H)$ to $\Phi(i)\Phi(j)\in E(G)$. For short, we set $\Phi(ij)=\Phi(i)\Phi(j)$. In this way, we can see $\Phi$ as a surjective map from $V(H)\cup E(H)$ to $V(G)\cup E(G)$. Therefore, we also get a surjective map $\tilde{\Phi}$ from $\C_H\cap \ZZ^{V(H)\cup E(H)}$ to $\C_G\cap\ZZ^{V(G)\cup E(G)}$. More  precisely, $a\in \ZZ^{V(H)\cup E(H)}$ is mapped to $b\in \ZZ^{V(H)\cup E(H)}$, where
        \[
        b_z=
        \begin{cases}
        a_v+a_w,\quad &\mbox{ if } z=u,\\
            a_{\Phi^{-1}(z)}, \quad &\mbox{ otherwise }
            
        \end{cases}.
        \]
        We note that $\tilde{\Phi}$ is well-defined since $\Phi$ is bijective when restricted to $(V(H)\setminus \{v,w\})\cup E(H)$ with image $(V(G)\setminus \{u\})\cup E(G)$. Moreover, since for every lattice point in $\C_H$ at least one of the coordinates $a_v$ and $a_w$ equals zero, we have $b_u=a_v$ or $b_u=a_w$. 

        As lattice points correspond to variables in the toric ring, we get a surjective ring homomorphism 
    $ \psi: K[\C_H]\to K[\C_G]$ sending the variable $x_\alpha$ to $x_{\tilde{\Phi}(\alpha)}$. In particular, note that $\psi(x_v)=\psi(x_w)=x_u$. Since every cycle in $G$ either corresponds to a cycle in $G_1$ or to a cycle in $G_2$ and as such to a cycle in $H$ (and vice versa) and since $\Phi$ is a bijection from $E(H)$ to $E(G)$, it follows from \Cref{thm: reduced gens}, that there is a natural $1-1$-correspondence between generators of $I_{\C_G}$ and generators of $I_{\C_H}$.  This implies that the kernel of $\psi$ is the ideal generated by the binomial $x_v-x_w$ and hence 
        \[
        R_G/I_{\C_G}\cong R_H/(I_H+\langle x_v-x_u\rangle).
        \]
        Since $K[\C_G]=K[\vw^\vp~:~ \vp\in \C_G\cap \ZZ^{V(G)\cup E(G)}]\cong R_G/I_{\C_G}$ (and similarly for $H$), $\psi$ also induces a surjective homomorphism between the monoid algebras (and hence also on the underlying monoids) $K[\C_H]$ and $K[\C_G]$ and by the above we have
        
        \[
       K[\C_G]\cong K[\C_H]/\langle \vw^{\e_v}-\vw^{\e_w}\rangle.
        \]
        Since under these homomorphisms the (variable corresponding to the) unit vectors of $v$ and $w$ have the same image (namely, the (variable corresponding to the) unit vector of $u$) and since $\dim \C_G=\dim \C_H-1$, we can conclude  with \cite[Theorem 2.1]{bruns} that
   \[
   \mathrm{Hilb}(K[\C_G];z)=\mathrm{Hilb}(K[\C_H];z)\cdot (1-z).
   \]
   The claim is now immediate from \Cref{prop:union}.
\end{proof}

\begin{example}
    \label{ex: loop graph}
    Let $L_m$ denote the \emph{$m$-loop graph}, which is the graph with a single vertex $i$ and $m$ loops at node $i$. 
    It is straightforward to verify that $h^\ast(\mathcal{C}_{L_1}; z) = 1 + z$. 
    It then follows from \Cref{thm: 1-sum} that 
    \[
    h^\ast(\mathcal{C}_{L_m}; z) = (1 + z)^m.
    \]
\end{example}

\begin{example}
    \label{ex: tree and forest graph}
    Let $P_1$ denote the path of length one, i.e., $P_1$ consists of two different vertices $u \neq v$, and the edge $uv$. It is easy to check that the corresponding cosmological polytope $\C_{P_1}$ has $h^\ast$-polynomial $h^\ast(\C_G;z)=1+3z$. By \Cref{thm: 1-sum}, this directly yields that the cosmological polytope $\C_{T_m}$ of a tree $T_m$ with $m$ edges has $h^\ast$-polynomial
    \begin{equation}\label{eq:tree}
        h^\ast(\C_{T_m};z)=(1+3z)^m.
    \end{equation}
    Evaluating $h^\ast(\C_{T_m};z)$ at $z=1$ gives that the normalized volume of $\C_{T_m}$ is given by $\mathrm{Vol}(\C_{T_m})=4^m$, which generalizes  \cite[Corollary 3.4]{jsl-cosmo} and reproves \cite[Corollary 4.2]{kuhne2022faces}.

If, more generally, $F_m$ is a forest on $m$ edges, then \eqref{eq:tree} and \Cref{prop:union} imply that
    \begin{align*}
        h^\ast(\C_{F_m};z)=(1+3z)^m.
    \end{align*}
   Hence, the $h^\ast$-polynomial of the cosmological polytope of any forest only depends on the number of its edges.
\end{example}

\subsection{Bounds for the \texorpdfstring{$h^\ast$}{h*}-polynomial}
\label{subsec: bounds}
In this section, we show that the $h^\ast$-polynomial $h^\ast(\C_G;z)$ of a graph $G=(V,E)$ with $k$ connected components
can be bounded coefficient-wise from below by a polynomial only depending on $k$, $\vert V \vert$ and $\vert E \vert$.
Given polynomials $f=\sum_{i=0}^na_iz^i$ and $g=\sum_{i=0}^mb_iz^i$, we write $f \preccurlyeq g$ if $f$ is coefficientwise smaller than $g$, i.e., $a_i \le b_i$ for all $i \in \mathbb{N}$. Here we set $a_i=0$ for $i > n$, and $b_i=0$ for $i > m$.

We can now state the main result of this section:

\begin{theorem}
    \label{thm: lower bound}
    Let $G=(V,E)$ be a graph with $k$ connected components. Then the $h^\ast$-polynomial of $\C_G$ satisfies
    \begin{align*}
        (1+3z)^{\vert V \vert-k}(1+z)^{\vert E \vert - \vert V \vert +k} \preccurlyeq h^\ast(\C_G;z).
    \end{align*}
    Equality is attained if and only if $G$ can be obtained as a $1$-sum of loops and a forest. 
\end{theorem}

The proof of this result will rely on the following lemma:

\begin{lemma}
    \label{lem: lower bound}
    Let $G=(V,E)$ be a graph
    and let $i,j\in V$. Let $G+f$ be the graph on vertex set $V$ and edge set $F \coloneqq E \cup \{f\}$, where $f \notin E$ is a new (possibly multi)edge between $i$ and $j$. Then
    \begin{align*}
        (1+z) h^\ast(\C_G;z) \preccurlyeq h^\ast(\C_{G +f} ;z).
    \end{align*}
    Moreover, equality holds if and only if $f$ is a loop. 
\end{lemma}

\begin{proof}
First assume that $i=j$, i.e., $f$ is a loop. In this case, applying \Cref{thm: 1-sum} and using \Cref{ex: loop graph} directly yields that
$ (1+z) h^\ast(\C_G;z) = h^\ast(\C_{G +f} ;z)$. 

Now assume that $i\neq j$. 
     Let $\widetilde{\C}_G \subseteq \RR^{V\cup F}$ be the canonical embedding of  $\C_G$ on the hyperplane $\mathcal{H}\coloneqq \{w\in \RR^{V\cup F}~:~w_f=0\}$ and let $V(\widetilde{\C_G})$ denote the set of vertices of $\widetilde{\C}_G$. Then,  $\C_{G+f}=\conv(V(\widetilde{\C_G}) \cup \{r,s,t\})$, where
    $r=\e_i+\e_j-\e_{f}$, $s=\e_i-\e_j+\e_{f}$ and  $t=-\e_i+\e_j+\e_{f}$.
     Since $\widetilde{\C_G} \subseteq\mathcal{H}$, and $r_{f}=-1$, and $s_{f}=1$, the polytope $P \coloneqq \conv(V(\widetilde{\C_G}) \cup \{r,s\})$ is a bipyramid over $\widetilde{\C}_G$ with apices $r$ and $s$, at height $-1$ and $1$, respectively.
As, in addition, the  midpoint $\frac{1}{2}(r+s)=\e_i$ of $r$ and $s$ lies in $\widetilde{\C}_G$, it follows (see e.g., \cite{ccd}) that $h^\ast(P;z)=(1+z)h^\ast(\widetilde{\C_G};z)$. Since $\conv(V(\widetilde{\C_G}) \cup \{r,s\}) \subseteq \C_{G+f}$, it follows from $h^\ast$-monotonicity \cite[Theoreom 3.3]{Stanley-monotonicity}, that 
    \begin{align*}
        (1+z)h^\ast(\C_G;z)=(1+z)h^\ast(\widetilde{\C_G};z) \preccurlyeq h^\ast(\C_{G+f};z).
    \end{align*}
    In particular, since the containment $\conv(V(\widetilde{\C_G}) \cup \{r,s\}) \subsetneq \C_{G+f}$ is strict, the above is a strict inequality, at least for some coefficients.
\end{proof}

We now proceed to the proof of \Cref{thm: lower bound}
\begin{proof}[Proof of \Cref{thm: lower bound}]
    First, assume that $G$ is connected. We can construct $G$ by iteratively adding edges to a spanning tree $T=(V,F) \subseteq G$ of $G$. As $T$ is a spanning tree, we have $\vert F\vert =\vert V\vert -1$ and it follows from \Cref{ex: tree and forest graph} that $h^\ast(T;z)=(1+3z)^{\vert V\vert -1}$. Applying \Cref{lem: lower bound} iteratively to the edges in $E\setminus F$ shows the claim.

    If $G$ is disconnected,
    then the previous argument can be applied to every connected component and the statement then follows from \Cref{prop:union}.
\end{proof}

\begin{remark}
    It follows from \Cref{thm: lower bound} that the normalized volume of the cosmological polytope of a graph $G = (V, E)$ with $k$ connected components
    is bounded from below by $4^{|V| - k}2^{|E| - |V| + k}$.
    This gives an exponential (and tight) lower bound on the number of computations needed to compute the wavefunction associated to $G$ according to a physical theory corresponding to a unimodular triangulation of $\C_G$ (i.e., via \Cref{eqn: canform decomp}).
    It appears that a corresponding (tight) upper bound of $4^{|E|}$ also holds, which may be shown by affirming the following conjecture.
\end{remark}

\begin{conjecture}
    \label{conj: upper bound}
    For a graph $G = (V, E)$
    we have
    \[
    h^\ast(\C_G; z) \preceq (1 + 3z)^{|E|}. 
    \]
\end{conjecture}

Note that by \Cref{ex: tree and forest graph}, the proposed upper bound in \Cref{conj: upper bound} is the $h^\ast$-polynomial of $\C_T$ where $T$ is a tree with the same number of edges as $G$.

\subsection{Degree of the \texorpdfstring{$h^\ast$}{h*}-polynomial}
\label{subsec: degree}
Let $P\subseteq \RR^n$ be a $d$-dimensional lattice polytope. 
The \emph{codegree} of $P$, denoted by $\codeg(P)$, is the smallest positive integer $t$ for which $tP$ contains a lattice point in its relative interior. 
It is well-known that the degree of $h^\ast(P;z)$ is equal to $d + 1- \codeg(P)$ (see, for instance, \citep{ccd}).
In the following, we compute the degree of the $h^\ast$-polynomial of a cosmological polytope. 
To do so, we make use of the following hyperplane description of $\C_G$ obtained by \cite{arkani2017cosmological}.

\begin{theorem}\citep{arkani2017cosmological}
\label{thm: facet description}
Let $G=(V,E)$ be a connected graph. The facets of $\C_G$ are in  bijection with the non-empty connected subgraphs of $G$. More precisely,  given a connected subgraph $H \subseteq G$ the hyperplane
\begin{align*}
   \left\{w\in \RR^{V\cup E}~:~ \sum_{v \in V(H)}w_v + \sum_{f \notin E(H)}\vert f \cap V(H) \vert w_f = 0\right\}
\end{align*}
defines a facet and each facet can be obtained in this way. 
\end{theorem}

With the help of Theorem~\ref{thm: facet description}, we may prove the following lemma.

\begin{lemma}
\label{contained in proper face}
Let $G=(V,E)$ be a connected, loopless graph with $n$ vertices and let $S\subseteq  \C_G \cap \ZZ^{V \cup E}$ with $\vert S\vert=n-1$. Then there is a proper face $F \subsetneq \C_G$, that contains $S$.
\end{lemma}

\begin{proof}
Let $S=\{\vp_1,\ldots,\vp_{n-1} \}$. 
By \Cref{thm: facet description}, in order to show the claim, it suffices to construct a connected subgraph of $G$ such that the corresponding facet contains all  points in $S$.

In the following, we distinguish two different types of lattice points in $\C_G$: Those, which are of the form $\e_i+\e_j-\e_f$, $\e_i-\e_j+\e_f$, $-\e_i+\e_j+\e_f$ or $\e_f$ for an edge $f=\{i,j\} \in E$, will be referred to as \emph{edge point}; those of the form  $\e_i$ for a vertex $i \in V$ will be referred to as \emph{vertex point}. 
Let $f_1,\ldots, f_k\in E$ be the edges of $G$ corresponding to the edge points in $S$ and 
let $G[f_1,\ldots,f_k]$ be the graph with vertex set $V$ and edge set $\{f_1,\ldots, f_k\}$. We show the following claim:

\medskip

{\sf Claim:} There exists a connected component $T=(V(T),E(T))$ of  $G[f_1,\ldots,f_k]$ that is a tree, such that $\{\e_i~:~i\in V(T)\}\cap  S=\emptyset$ and $S$ contains a  unique edge point for any $f\in E(T)$. Note that $E(T)$ may be empty, i.e., $H$ can consist of an isolated vertex.

If, by contradiction, the claim is false, then every connected component of $G[f_1,\ldots,f_k]$ contains a cycle, or a vertex $i$ with $\e_i \in S$, or an edge with at least two of its edge points belonging to $S$.  In each of these three cases, the connected component corresponds to at least as many lattice points in $S$ as it has vertices. (Note that since each connected component of $G[f_1,\ldots,f_k]$ with $\ell$ vertices contains at least $\ell-1$ edges, it corresponds to at least $\ell-1$ edge points in $S$.) Since every vertex of $V$ belongs to a unique component, it follows that $\vert S\vert \geq n$, which yields a contradiction.

\medskip

Fix a tree $T=(V(T),E(T)) \subseteq G$ as in the above claim. 
Pick any vertex $i' \in V(T)$ and let $T' \subseteq T$ denote the unique subtree of $T$ that contains $i'$ and every vertex $j'$ that can be reached from $i'$ by only using edges $f=ij \in E(T)$ whose corresponding edge points that are contained in $S$ are of the form $\e_f$ or $\e_{i}-\e_{j}+\e_f$ or $-\e_{i}+\e_{j}+\e_f$.
Note that if $\vp \in S$ is an edge point of an edge $f=ij \in E$ such that $\vert V(T')\cap\{i,j\}\vert=1$, then, by construction, this implies that $\vp=\e_{i}+\e_{j}-\e_f$.
By \Cref{thm: facet description}, the tree $T'$ gives rise to a facet ${F}_{T'}$ of $\C_G$ with supporting hyperplane
\begin{align}
    \label{eq: facet equality}
    \left\{w\in \RR^{V\cup E}~:~\sum_{v \in V(T')}w_v + \sum_{f \notin E(T')}\vert f \cap V(T')\vert w_f=0\right\}.
\end{align}

To finish the proof, we show that $S \subseteq F_{T'}$. If $\vp \in S$ is a vertex point $\e_{i}$ for some $i \in V$. Then $i \notin V(T')$ by construction of $T'$ and, hence, $\vp \in {F}_{T'}$. Suppose now,  $\vp\in S$ is an edge point of the edge $f=ij \in E$. If $i,j \notin V(T')$, it is immediate that $\vp \in {F}_{T'}$. If $\vert\{i,j\}\cap V(T')\vert=1$, then $\vp=\e_i+\e_j-\e_f$ as seen above. But this means that $\vp$ satisfies \Cref{eq: facet equality}, and therefore, $\vp \in {F}_{T'}$. Finally, assume that  $\{i,j\}\subseteq  V(T')$. By the construction of $T'$ this implies that $f \in E(T')$. Moreover, we have $\vp=\e_f$ or $\vp=-\e_i+\e_j+\e_f$ or $\vp=\e_i-\e_j+\e_f$ and in all these cases it is immediate that $\vp$ fulfills \Cref{eq: facet equality}, and, hence, $\vp \in {F}_{T'}$.
\end{proof}

Lemma~\ref{contained in proper face}  yields the following theorem.

\begin{theorem}
\label{thm: degree}
  For any graph $G=(V,E)$, we have
    \begin{align*}
        \codeg(\C_G)=\vert V\vert \qquad \text{and} \qquad \deg(h^\ast(\C_G;z))=\vert E \vert.
    \end{align*}
\end{theorem}

\begin{proof}
First assume that $G$ is connected and loopless. 
As, by \citep{jsl-cosmo}, the polytope $\C_G$ admits a regular unimodular triangulation, it  follows from  \cite[Proposition 2.60]{bruns2009gubeladze} that $\C_G$ has the integer decomposition property (IDP), meaning that every lattice point $\vp \in t\C_G \cap \ZZ^{V \cup E}$ can be written as
$\vp=\vp_1+\cdots+\vp_t$, with $\vp_i \in \C_G \cap \ZZ^{V \cup E}$ for $i \in [t]$.
For $t \le \vert V \vert-1$, \Cref{contained in proper face} implies there is a proper face $F \subsetneq \C_G$ containing the points $\vp_1,\ldots,\vp_t$, and in particular, we have $\frac{1}{t}\vp \in F$. Thus, $\vp$ is not in the interior of $t\C_G$ and we have $\codeg(\C_G) \ge \vert V \vert$.

Now, consider the point $\vp \coloneqq \sum_{i \in V}\e_i$ which is contained in $\vert V \vert\cdot \C_G$. By \Cref{thm: facet description}, it is immediate that $\vp$ is not contained in any facet of $\vert V \vert \cdot \C_G$, which shows $\codeg(\C_G)=\vert V \vert$.
This directly implies $\deg(h^\ast(\C_G;z))=\dim(\C_G)+1-\codeg(\C_G)=\vert E\vert$.

If $G$ has loops, then applying \Cref{thm: 1-sum} and using \Cref{ex: loop graph} and the statement for connected graphs, shows the claim. 
If $G$ is disconnected, then the statement follows from the connected case together with \Cref{prop:union}. 
\end{proof}

A lattice polytope is called \emph{Gorenstein (of index~$s$)} if its associated toric ring is Gorenstein of index $s$. 
\cite{stanley1978hilbert} proved that a lattice polytope $P$ is Gorenstein of index $s$ if and only if it has codegree $s$ and $h^\ast_i(P) = h^\ast_{d+1 - s - i}(P)$ for all $0\leq i\leq  d+1-s$. 
Since $h^\ast_0(P) = 1$, being Gorenstein of index $s$ implies that  $h^\ast_{d+1 -s}(P) = 1$. 
With the help of this simple observation, we may characterize the Gorenstein cosmological polytopes.

\begin{theorem}
\label{thm: gorenstein}
    The cosmological polytope $\C_G$ of a graph $G$ is Gorenstein, if and only if every edge of $G$ is a loop.
\end{theorem}

\begin{proof}
If $G$ is a graph whose edges are all loops, then, using \Cref{ex: loop graph} and \Cref{prop:union}, we have $h^\ast(\C_G;z) = (1+z)^{|E|}$. Thus, $\C_G$ is Gorenstein.

Suppose now that there is at least one edge $f \in E$, that is not a loop. By \Cref{thm: lower bound}, we can bound $h^\ast(\C_G,z)$ from below by the following polynomial of degree $\vert E\vert$
\[
 (1+3z)^{\vert V \vert-k}(1+z)^{\vert E \vert - \vert V \vert +k} .
\]
By \Cref{thm: degree}, $\deg(h^\ast(\C_G;z))=\vert E\vert$, which implies that the leading coefficient of $h^\ast(\C_G;z)$ is at least $3$, which already contradicts that $\C_G$ is Gorenstein as mentioned above. 
\end{proof}

\section{\texorpdfstring{$h^\ast$}{h*}-polynomials of multicycles and multitrees}
\label{sec: examples}

We now provide explicit formulas for the $h^\ast$-polynomials of cosmological polytopes $\C_G$, when the underlying graph $G$ is a tree or a cycle with possibly multiple edges. 
As a consequence, we obtain generalizations of the normalized volume formulas for simple trees and cycles obtained in \citep{kuhne2022faces} and \citep{jsl-cosmo}, respectively. 
Our results rely on the identification of half-open decompositions of certain unimodular triangulations extracted from the Gr\"obner bases of $I_{\C_G}$ of \cite{jsl-cosmo}. 
We refer the reader to \Cref{sec:ehrhart} and \Cref{sec:unimodular} for relevant preliminaries. 

In the following, given a graph $G=(V,E)$ and vertices $i$ and $j$ in $V$, the multiset of \emph{all} edges with endpoints $i$ and $j$ will be called a \emph{multi-edge}. 
The cardinality of this multiset will be referred to as the \emph{multiplicity} of the multi-edge.

\subsection{Multicycles}
\label{subsec: cycles}
Given an integer $n\geq 3$ and $\mathbf{a}=(a_1,\ldots,a_n)\in \mathbb{N}^n$ with $a_i\geq 1$ for all $1\leq i\leq n$, we let  $C_{a_1,\ldots,a_n}$  (or $C_{\mathbf{a}}$ for short) denote the graph on vertex set $[n]$ and whose edge set $E(C_\mathbf{a})$ consists of a multi-edge of multiplicity $a_i$ between node $i$ and $i+1$  for every $1\leq i\leq n$. Here and in the following, we consider $i$ $\mathrm{mod}$ $n$. Note that the graph $C_{(1,\ldots,1)}$ is just the usual $n$-cycle.

\begin{figure}
\begin{center}
\begin{tikzpicture}[scale=2]
    \foreach \i in {0,...,7} {
        \node[] (\i) at (\i*360/8:1) {};
        \draw[fill] (\i) circle (1.5pt);
    }
    \node[above] at (2) {\small$1$};
    \node[above left] at (3) {\small$n$};
    \node[left] at (4) {\small$n-1$};
    \node[below left] at (5) {\small$6$};
    \node[below] at (6) {\small$5$};
    \node[below right] at (7) {\small$4$};
    \node[right] at (0) {\small$3$};
    \node[above right] at (1) {\small$2$};

    \draw (2) to (1);
    \draw (2) to [bend left] (1);
    \draw (2) to [bend left=60] (1);
    \draw (2) to [bend right=60] (1);
    \node (e1a1) at (67.5:0.5) {\tiny$e_{a_1}^{(1)}$};
    \node[rotate=67.5] (dots) at (67.5:0.82) {\tiny$\dots$};
    \node (e11) at (67.5:1.3) {\tiny$e_1^{(1)}$};

    \draw (1) to (0);
    \draw (1) to [bend left] (0);
    \draw (1) to [bend left=60] (0);
    \draw (1) to [bend right=60] (0);
    \node (e2a2) at (22.5:0.5) {\tiny$e_{a_2}^{(2)}$};
    \node[rotate=22.5] (dots) at (22.5:0.82) {\tiny$\dots$};
    \node (e21) at (22.5:1.3) {\tiny$e_1^{(2)}$};

    \draw (0) to (7);
    \draw (0) to [bend left] (7);
    \draw (0) to [bend left=60] (7);
    \draw (0) to [bend right=60] (7);
    \node (e3a3) at (-22.5:0.5) {\tiny$e_{a_3}^{(3)}$};
    \node[rotate=-22.5] (dots) at (-22.5:0.82) {\tiny$\dots$};
    \node (e31) at (-22.5:1.3) {\tiny$e_1^{(3)}$};

    \draw (7) to (6);
    \draw (7) to [bend left] (6);
    \draw (7) to [bend left=60] (6);
    \draw (7) to [bend right=60] (6);
    \node (e4a4) at (-67.5:0.5) {\tiny$e_{a_4}^{(4)}$};
    \node[rotate=-67.5] (dots) at (-67.5:0.82) {\tiny$\dots$};
    \node (e41) at (-67.5:1.3) {\tiny$e_1^{(4)}$};

    \draw (6) to (5);
    \draw (6) to [bend left] (5);
    \draw (6) to [bend left=60] (5);
    \draw (6) to [bend right=60] (5);
    \node (e5a5) at (-112.5:0.5) {\tiny$e_{a_5}^{(5)}$};
    \node[rotate=-112.5] (dots) at (-112.5:0.82) {\tiny$\dots$};
    \node (e51) at (-112.5:1.3) {\tiny$e_1^{(5)}$};

    \node[rotate=112.5] (dots) at (202.5:0.95) {$\dots$};

    \draw (4) to (3);
    \draw (4) to [bend left] (3);
    \draw (4) to [bend left=60] (3);
    \draw (4) to [bend right=60] (3);
    \node (en-1an-1) at (157.5:0.5) {\tiny$e_{a_{n-1}}^{(n-1)}$};
    \node[rotate=157.5] (dots) at (157.5:0.82) {\tiny$\dots$};
    \node (en-11) at (157.5:1.3) {\tiny$e_1^{(n-1)}$};

    \draw (3) to (2);
    \draw (3) to [bend left] (2);
    \draw (3) to [bend left=60] (2);
    \draw (3) to [bend right=60] (2);
    \node (enan) at (112.5:0.5) {\tiny$e_{a_n}^{(n)}$};
    \node[rotate=112.5] (dots) at (112.5:0.82) {\tiny$\dots$};
    \node (en1) at (112.5:1.3) {\tiny$e_1^{(n)}$};
\end{tikzpicture}  
\caption{A multicycle} \label{fig:multicycle}
\end{center}
\end{figure}
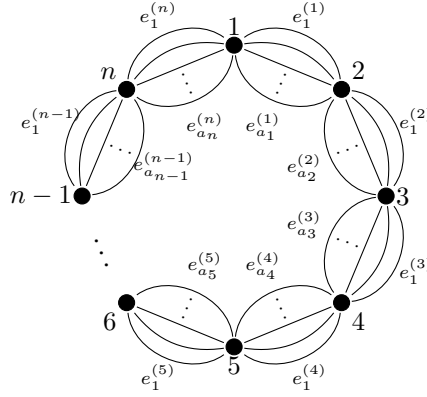

For $1\leq i\leq n$ and $1\leq j\leq a_i$ we use $e_j^{(i)}$ to denote the $j$-th edge between $i$ and $i+1$. In all our figures, we will put the vertices of $C_{\mathbf{a}}$ in clockwise order and label the edges $e_j^{(i)}$ from external to internal by increasing lower index (see \Cref{fig:multicycle}). If we just draw part of a cycle, the labeling is always from top to bottom by increasing lower index. We will also refer to directed edges from $i$ to $i+1$ and from  $i+1$ to $i$, respectively, as \emph{clockwise} and \emph{counter-clockwise} oriented edges, respectively. 

In order to be able to use \Cref{thm: visibility formula for h star} we first need to compute an explicit unimodular triangulation of the cosmological polytope of $C_{\mathbf{a}}$. Following \cite{jsl-cosmo} (see \Cref{thm:GB} and \Cref{cor: unimodular triangulation}) this requires the specification of a  \emph{good} term order $<$ on  $R_{\mathbf{a}} \coloneqq R_{C_{\mathbf{a}}}$. The obtained triangulation will generalize \cite[Theorem 4.1]{jsl-cosmo} that treats the case of a simple cycle. To keep the notation as simple as possible, given an edge $e_j^{(i)}$ ($1\leq i\leq n$ and $1\leq j\leq a_i$), we write
\begin{itemize}
    \item[$\bullet$] $\overrightarrow{y}_{j}^{(i)}$ for the variable $y_{i,i+1,e_j^{(i)}}$,
    \item[$\bullet$] $\overleftarrow{y}_{j}^{(i)}$ for the variable $y_{i+1,i,e_j^{(i)}}$,
    \item[$\bullet$] $t_{j}^{(i)}$ for the variable $t_{e_j^{(i)}}$
    \item[$\bullet$] $z_{j}^{(i)}$ for the variable $z_{e_j^{(i)}}$ and, as usual, $z_i$.
\end{itemize}
We write $V(e_j^{(i)})$ for the set of variables associated to $e_j^{(i)}$, i.e., $V(e_j^{(i)}) \coloneqq \{\overrightarrow{y}_{j}^{(i)},\overleftarrow{y}_{j}^{(i)},t_{j}^{(i)},z_{j}^{(i)}\}$.

We order the variables of $R_{\mathbf{a}}$ as follows:
\begin{align*} 
&\overrightarrow{y}_{1}^{(1)} & > & \overrightarrow{y}_{2}^{(1)} & > & \cdots & > & \overrightarrow{y}_{a_1}^{(1)} & > & \overrightarrow{y}_{1}^{(2)} & > & \cdots & > & \overrightarrow{y}_{a_2}^{(2)} & > & \cdots & > &\overrightarrow{y}_{1}^{(n)} & > & \cdots & > & \overrightarrow{y}_{a_n}^{(n)}\\
   > &\overleftarrow{y}_{a_n}^{(n)} & > & \overleftarrow{y}_{a_n-1}^{(n)} & > & \cdots & > & \overleftarrow{y}_{1}^{(n)} & > & \overleftarrow{y}_{a_{n-1}}^{(n-1)} & > & \cdots & > & \overleftarrow{y}_{1}^{(n-1)} & > & \cdots & > & \overleftarrow{y}_{a_1}^{(1)} & > & \cdots & > & \overleftarrow{y}_{1}^{(1)} \\
   >&z_{1}^{(1)} & > & z_{2}^{(1)} & > & \cdots & > & z_{a_1}^{(1)} & > & z_{1}^{(2)} & > & \cdots & > & z_{a_2}^{(2)} & > & \cdots & > &z_{1}^{(n)} & > & \cdots & > & z_{a_n}^{(n)} \\
   >&t_{1}^{(1)} & > & t_{2}^{(1)} & > & \cdots & > & t_{a_1}^{(1)} & > & t_{1}^{(2)} & > & \cdots & > & t_{a_2}^{(2)} & > & \cdots & > &t_{1}^{(n)} & > & \cdots & > & t_{a_n}^{(n)} \\
   > & z_1 & > & \cdots & > & z_n, &&&&&&&&&&&&&&&
\end{align*}
and then use the lexicographic term order with respect to the above ordering on $R$. Recall that given monomials $f,g \in R$, we have $f < g$ in lexicographic order if and only if the exponent of the greatest variable that has a different exponent in $f$ than in $g$, is greater in $g$ than in $f$. It is immediate that $<$ is indeed a good term order. We let $\mathcal{T}_{\mathbf{a}}$ be the triangulation from \Cref{cor: unimodular triangulation} with respect to this term order.  In the following, we list the obstructions for a set $S$ to be a facet of $\mathcal{T}_{\mathbf{a}}$  where we use the graphical representations introduced in \Cref{sec:unimodular}. To avoid confusion between different edges of a multi-edge and different types of edges for the same edge, we will draw a dotted frame if the \emph{same} edge occurs in different types. 

\noindent \textit{Fundamental obstructions:} The leading terms of the fundamental binomials correspond to the following forbidden subgraphs for each edge:
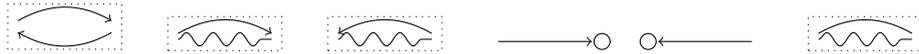
\begin{figure}[h]
    \centering
    \begin{tabular}{cccccc}
    \begin{tikzpicture}
            \node[] (0) at (0, 0) {};
            \node[] (1) at (1.5, 0) {};
            \draw[->] (0) to [bend left] (1);
            \draw[<-] (0) to [bend right] (1);
            \draw[dotted] (0,-0.3) rectangle (1.5,0.3);
        \end{tikzpicture} & 
        \begin{tikzpicture}
            \node[] (0) at (0, 0) {};
            \node[] (1) at (1.5, 0) {};
            \draw[->] (0) to [bend left] (1);
            \draw[squiggle] (0) to (1);
            \draw[dotted] (0,-0.15) rectangle (1.5,0.3);
        \end{tikzpicture} &
        \begin{tikzpicture}
            \node[] (0) at (0, 0) {};
            \node[] (1) at (1.5, 0) {};
            \draw[<-] (0) to [bend left] (1);
            \draw[squiggle] (0) to (1);
            \draw[dotted] (0,-0.15) rectangle (1.5,0.3);
        \end{tikzpicture} &
        \begin{tikzpicture}
            \node[] (0) at (0, 0) {};
            \node[] (1) at (1.5, 0) {};
            \draw[white] (1) circle (3pt);
            \draw[->] (0) to (1);
        \end{tikzpicture} &
        \begin{tikzpicture}
            \node[] (0) at (0, 0) {};
            \draw[white] (0) circle (3pt);
            \node[] (1) at (1.5, 0) {};
            \draw[<-] (0) to (1);
        \end{tikzpicture} &
        \begin{tikzpicture}
            \node[] (0) at (0, 0) {};
            \node[] (1) at (1.5, 0) {};
            \draw (0) to [bend left] (1);
            \draw[squiggle] (0) to (1);
            \draw[dotted] (0,-0.15) rectangle (1.5,0.3);
        \end{tikzpicture}
\end{tabular}
    \caption{Fundamental obstructions}
    \label{fig: fundamental obstructions}
\end{figure}

In the above figures, the same edge is depicted and not a multi-edge of $C_\textbf{a}$.

\noindent\textit{Zig-zag binomials:} Leading terms of zig-zag binomials correspond to paths of length at least two, consisting of undirected and clockwise directed edges only, where the last node in clockwise direction is white, the last edge in clockwise direction is undirected and the first edge in clockwise direction is directed clockwise. Such  a subgraph will be called \emph{partially directed path in clockwise direction}. An example for such a path is the following
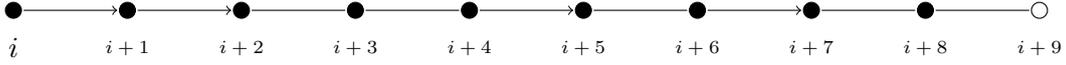
\begin{center}
    \begin{tikzpicture}
        \foreach \i in {0,...,8} {
        \node[] (\i) at (\i*1.5,0) {};
        \draw[fill] (\i) circle (3pt);
        }
        \node[] (9) at (9*1.5,0) {};
        \draw[white] (9) circle (3pt);
        \draw[->] (0) to (1) {};
        \draw[->] (1) to (2) {};
        \draw[] (2) to (3) {};
        \draw (3) to (4) {};
        \draw[->] (4) to (5) {};
        \draw (5) to (6) {};
        \draw[->] (6) to (7) {};
        \draw (7) to (8) {};
        \draw (8) to (9) {};
        \foreach \i in {1,...,9} {
        \node[] (\i) at (\i*1.5,-0.5) {{\tiny $i+\i$}};
        }
        \node[] (0) at (0,-0.5) {$i$};
    \end{tikzpicture}
\end{center}

\noindent \textit{Cyclic binomials:} There are two different types of cycles in $C_{\mathbf{a}}$:  Cycles of length $2$, consisting of two edges of a multi-edge, and cycles of length $n$, containing exactly one edge per  multi-edge of $C_{\mathbf{a}}$. The leading terms of the first type correspond to the four forbidden subgraphs:

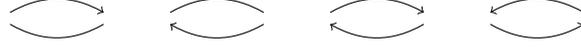
\begin{figure}[h]
\begin{center}
\begin{tabular}{cccc}
    \begin{tikzpicture}
            \node[] (0) at (0, 0) {};
            \node[] (1) at (1.5, 0) {};
            \draw[->] (0) to [bend left] (1);
            \draw (0) to [bend right] (1);
            \node[] (0) at (0,-0.5) {};
            \node[] (1) at (1.5,-0.5) {};
        \end{tikzpicture} & 
        \begin{tikzpicture}
            \node[] (0) at (0, 0) {};
            \node[] (1) at (1.5, 0) {};
            \draw (0) to [bend left] (1);
            \draw[<-] (0) [bend right] to (1);
            \node[] (0) at (0,-0.5) {};
            \node[] (1) at (1.5,-0.5) {};
        \end{tikzpicture} &
        \begin{tikzpicture}
            \node[] (0) at (0, 0) {};
            \node[] (1) at (1.5, 0) {};
            \draw[->] (0) to [bend left] (1);
            \draw[<-] (0) [bend right] to (1);
            \node[] (0) at (0,-0.5) {};
            \node[] (1) at (1.5,-0.5) {};
        \end{tikzpicture} &
        \begin{tikzpicture}
            \node[] (0) at (0, 0) {};
            \node[] (1) at (1.5, 0) {};
            \draw[<-] (0) to [bend left] (1);
            \draw[->] (0) [bend right] to (1);
            \node[] (0) at (0,-0.5) {};
            \node[] (1) at (1.5,-0.5) {};
        \end{tikzpicture} 
\end{tabular}
\end{center}
\caption{Cyclic obstructions}
    \label{fig:cyclic}
\end{figure}
In the above figures, two edges of the same multi-edge are depicted with the one with smaller index on top. The forbidden subgraphs coming from leading terms of cycles of length $n$ consist of either a clockwise directed or an undirected edge for every multi-edge of $C_{\mathbf{a}}$ with at least one directed edge. Additionally, we have that the cycles of length $n$ consisting of exactly one counterclockwise directed edge for every multi-edge of $C_{\mathbf{a}}$ are forbidden. Subgraphs  with at least one undirected edge are called \emph{partially directed cycles in clockwise direction} and subgraphs  with only directed edges are called \emph{directed cycles in (counter)clockwise direction} depending on the direction of the appearing edges. 

For characterizing the subsets of the generators of $R_{\mathbf{a}}$ that define facets of $\Ta$, we first characterize how facets look locally when restricted to a certain multi-edge.
\begin{lemma}
    \label{lem: induced G_S on one multi edge}
    Let $S\subseteq R_{\mathbf{a}}$ be a facet of $\Ta$ and $1\leq i \leq n$. Then one of the following cases occurs:
    \begin{enumerate}
    \item[(A)] There exists a unique $1\leq j\leq a_i$ such that $S \cap V(e_j^{(i)})=\{\overrightarrow{y}_{j}^{(i)},z_{j}^{(i)}\}$ and for $k \in [a_i]\setminus\{j\}$, we have $\vert S \cap  V(e_k^{(i)})\vert=1$. More precisely,
    \begin{align*}
            S \cap V(e_k^{(i)}) \subseteq 
            \begin{cases}
                \{t_{k}^{(i)},z_{k}^{(i)}\}, &\text{ if } k < j, \\
                \{t_{k}^{(i)},\overrightarrow{y}_{k}^{(i)}\}, &\text{ if } k > j. 
            \end{cases}
            \end{align*}
     \item[(B)] There exists a unique $1\leq j\leq a_i$ such that $S \cap V(e_j^{(i)})=\{\overleftarrow{y}_{j}^{(i)},z_{j}^{(i)}\}$ and for $k \in [a_i]\setminus\{j\}$, we have $\vert S \cap  V(e_k^{(i)})\vert=1$. More precisely,
    \begin{align*}
            S \cap V(e^{(i)}_k) \subseteq 
            \begin{cases}
                \{t_{k}^{(i)},\overleftarrow{y}_{k}^{(i)}\}, &\text{ if } k < j, \\
                \{t_{k}^{(i)},z_{k}^{(i)}\}, &\text{ if } k > j.
            \end{cases}
            \end{align*}
            \item[(C)] $\vert S \cap V(e_j^{(i)})\vert=1$ for all $1\leq j\leq a_i$. 
     \end{enumerate}
    
   \end{lemma}

\begin{proof}
   It was already shown in Proposition 2.12 and Lemma 2.11 of \cite{jsl-cosmo} that $1 \le \vert S \cap V(e_j^{(i)})\vert \leq 2$ and if $\vert S \cap V(e_j^{(i)})\vert =2$, then $S \cap V(e_j^{(i)})$ is one of $\{\overrightarrow{y}_{j}^{(i)},z_{j}^{(i)}\}$ and $\{\overleftarrow{y}_{j}^{(i)},z_{j}^{(i)}\}$.

   Suppose we are not in situation (C), i.e., $\vert S \cap V(e_j^{(i)})\vert =2$ for some $1\leq j\leq a_i$. If $S \cap V(e_j^{(i)})=\{\overrightarrow{y}_{j}^{(i)},z_{j}^{(i)}\}$, then it directly follows that 
   $S \cap  V(e_k^{(i)})\subseteq \{t_{k}^{(i)},z_{k}^{(i)}\}$ if $k<j$ and  $S \cap  V(e_k^{(i)})\subseteq  \{t_{k}^{(i)},\overrightarrow{y}_{k}^{(i)}\}$ if $k>j$ since otherwise $S$ contains one of the obstructions in \Cref{fig:cyclic}. But then, by the first part of the proof, we must have $\vert S \cap  V(e_k^{(i)})\vert=1$ for all $k\in [a_i]\setminus \{j\}$. This is situation (A). If $S \cap V(e_j^{(i)})=\{\overleftarrow{y}_{j}^{(i)},z_{j}^{(i)}\}$, then a similar argument shows that we are in situation (B).
\end{proof}

 The above lemma says that if $S$ is a facet of $\Ta$, then each edge $e_j^{(i)}$ of $C_{\mathbf{a}}$ appears in $G_S\coloneqq (C_{\mathbf{a}})_S$ at least once and at most twice. In the latter case, $G_S$ can, locally, only look as follows
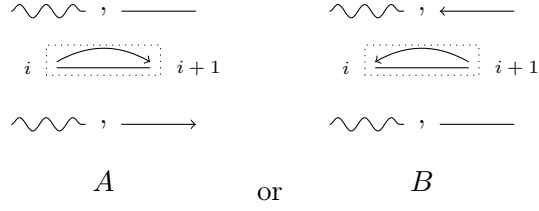
\begin{center}
    \begin{tabular}{ccc}
         \begin{tikzpicture}
            \node[] (0) at (0, 0) {};
            \node[] (1) at (1.5, 0) {};
            \node[] (a) at (0.75,0.75) {,};
            \node[] (b) at (0.75,-0.75) {,};
            \node[] (2) at (-0.6,0.75) {};
            \node[] (3) at (0.65,0.75) {};
            \node[] (4) at (0.85,0.75) {};
            \node[] (5) at (2.1,0.75) {};
            \node[] (6) at (-0.6,-0.75) {};
            \node[] (7) at (0.65,-0.75) {};
            \node[] (8) at (0.85,-0.75) {};
            \node[] (9) at (2.1,-0.75) {};
            
            \draw[squiggle] (2) to (3);
            \draw (4) to (5);
            \draw[->] (0) to [bend left] (1);
            \draw (0) to (1);
            \draw[squiggle] (6) to (7);
            \draw[->] (8) to (9);
            
            \node at (0.75, -1.5) {$A$};
            \node (i) at (-0.25,0) {\tiny{$i$}};
            \node (i+1) at (2,0) {\tiny{$i+1$}};
            \draw[dotted] (0,-0.1) rectangle (1.5,0.3);
        \end{tikzpicture}
        & or &
        \begin{tikzpicture}
            \node[] (0) at (0, 0) {};
            \node[] (1) at (1.5, 0) {};
            \node[] (a) at (0.75,0.75) {,};
            \node[] (b) at (0.75,-0.75) {,};
            \node[] (2) at (-0.6,0.75) {};
            \node[] (3) at (0.65,0.75) {};
            \node[] (4) at (0.85,0.75) {};
            \node[] (5) at (2.1,0.75) {};
            \node[] (6) at (-0.6,-0.75) {};
            \node[] (7) at (0.65,-0.75) {};
            \node[] (8) at (0.85,-0.75) {};
            \node[] (9) at (2.1,-0.75) {};
            
            \draw[squiggle] (2) to (3);
            \draw[<-] (4) to (5);
            \draw[<-] (0) to [bend left] (1);
            \draw (0) to (1);
            \draw[squiggle] (6) to (7);
            \draw (8) to (9);
            
            \node at (0.75, -1.5) {$B$};
            \node (i) at (-0.25,0) {\tiny{$i$}};
            \node (i+1) at (2,0) {\tiny{$i+1$}};
            \draw[dotted] (0,-0.1) rectangle (1.5,0.3);
        \end{tikzpicture}
    \end{tabular}
\end{center}
These pictures should be read as follows: As before, the frame means that the same edge, say $e_j^{(i)}$, appears in two types (undirected and directed) and we will call such an edge a \emph{double} edge. The edges drawn above (below) this double edge, separated by commas, mean that edges $e_k^{(i)}$ with $k<j$ ($k>j$) always occur in exactly one of the two types depicted. We will refer to such edges as \emph{single} edges.
A subgraph as depicted on the left and on the right, respectively, will be called a multi-edge of \emph{type A} and of \emph{type B}, respectively. A multi-edge with only single edges will be called a multi-edge of \emph{type C}. 
With this we can characterize the subsets $S$ of the generators of $R_\textbf{a}$ that are facets of the triangulation $\Ta$ in terms of their graphs $G_S$, generalizing \cite[Theorem 4.1]{jsl-cosmo}. Following the notation therein, we let $Z_S \coloneqq S \cap \{z_i \mid i \in [n]\}$ be the set of white nodes in $S$ and we write $G_{i,j}$ for the induced subgraph of $G_S$ on vertex set  $\{i,i+1,\ldots ,j\}$ (where we take the vertices $\mathrm{mod}\; n$), i.e., $G_{i,j}$ is the graph consisting of every vertex and every multi-edge of $G_S$ being passed while going from $i$ to $j$ in clockwise direction.

\begin{theorem}
\label{thm: cycle triangulation}
Let $S$ be a subset of the generators of $R_{\mathbf{a}}=R_{C_{\mathbf{a}}}$ and let $Z_S =\{z_{i_1}, \ldots , z_{i_k}\}$, where $i_1 < \cdots < i_k$. Then $S$ corresponds to a facet of the triangulation $\Ta$ of $\mathcal{C}_{C_{\mathbf{a}}}$ if and only if $Z_S \neq \emptyset$, and,
     for $1\leq t\leq k$, the graph $G_{i_t,i_{t+1}}$ is of one of the following two types: 
    
    \begin{enumerate}
        \item going from $i_t$ to $i_{t+1}$ in clockwise direction, the first multi-edge is of type $A$, followed by (possibly zero) multi-edges of type $A$ or $B$, followed by one multi-edge of type $C$ whose edges are directed counterclockwise or squiggly, and all (possibly zero) remaining multi-edges are of type $B$.
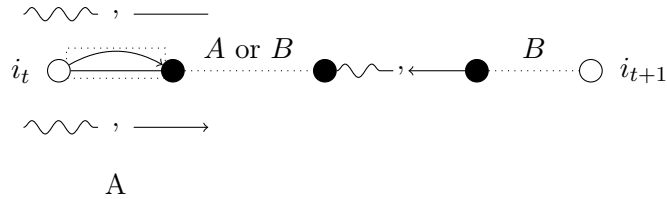
        \begin{center}
        \begin{tikzpicture}
            
            \node[black, inner sep=3pt] (1) at (1.5, 0) {};
            \node[] (a) at (0.75,0.75) {,};
            \node[] (b) at (0.75,-0.75) {,};
            \node[] (2) at (-0.6,0.75) {};
            \node[] (3) at (0.65,0.75) {};
            \node[] (4) at (0.85,0.75) {};
            \node[] (5) at (2.1,0.75) {};
            \node[] (6) at (-0.6,-0.75) {};
            \node[] (7) at (0.65,-0.75) {};
            \node[] (8) at (0.85,-0.75) {};
            \node[] (9) at (2.1,-0.75) {};

            \draw[dotted] (0.1,-0.1) rectangle (1.4,0.3);
            \node[white,inner sep=3pt] (0) at (0, 0) {};
            
            \draw[squiggle] (2) to (3);
            \draw (4) to (5);
            \draw[->] (0) to [bend left] (1);
            \draw (0) to (1);
            \draw[squiggle] (6) to (7);
            \draw[->] (8) to (9);
            \node at (0.75, -1.5) {A};
            
            \node[black,inner sep=3pt] (10) at (3.5, 0) {};
            \node[black,inner sep=3pt] (11) at (5.5, 0) {};
            \node[white,inner sep=3pt] (12) at (7, 0) {};
            \node[] (it) at (-0.5,0) {$i_t$};
            \node[] (it+1) at (7.7,0) {$i_{t+1}$};

            \node[] (13) at (2.5,0.3) {$A$ or $B$};
            \node[] (14) at (6.25,0.3) {$B$};
            
            \draw[dotted] (1) to (10);
            
            \draw[squiggle] (10) to ($(10)+(0.9,0)$);
            \draw[<-] ($(10) + (1.1,0)$) to (11);
            \node[] (15) at (4.5,0) {,};

            \draw[dotted] (11) to (12);

        \end{tikzpicture}
        \end{center}

        \item going from $i_t$ to $i_{t+1}$ in clockwise direction, the first edge is a multi-edge of type $C$, whose edges are undirected or squiggly, and all (possibly zero) remaining edges are of type $B$.
        
        \begin{center}
        \begin{tikzpicture}
            \node[white,inner sep=3pt] (0) at (0, 0) {};
            \node[black,inner sep=3pt] (1) at (2, 0) {};
            \node[white,inner sep=3pt] (2) at (3.5, 0) {};

            \node[] (3) at (2.75,0.3) {$B$};

            \draw[squiggle] (0) to (0.9,0);
            \draw[] (1.1,0) to (1);
            \node[] (4) at (1,0) {,};

            \draw[dotted] (1) to (2);

            \node[] (it) at (-0.5,0) {$i_{t}$};
            \node[] (it+1) at (4.2,0) {$i_{t+1}$};
        \end{tikzpicture}
        \end{center}
    \end{enumerate}  

\end{theorem}

\begin{proof}
    Using \Cref{lem: induced G_S on one multi edge}, it is immediate to check that a graph $G_S$ as described above does not contain any forbidden subgraphs. Since it is easily seen that $|S|=\dim \C_{C_{\mathbf{a}}}+1$, it follows that $G_S$ corresponds to a facet of $\Ta$.

    Assume now that $S$ is a facet and suppose that $Z_S=\emptyset$. 
    Since each multi-edge has at most one double edge due to \Cref{lem: induced G_S on one multi edge} and $\vert S \vert = \vert V \vert + \vert E \vert = n+a_1+\cdots+a_{n}$,  each multi-edge is forced to have exactly one double edge. 
    However, in this case, $G_S$ would contain a forbidden subgraph corresponding to the leading term of a cyclic binomial. 
    Hence, $Z_S \neq \emptyset$. Fix $1\leq t\leq k$. We show the following claim:

    \noindent{\sf Claim:} $G_{i_t,i_{t+1}}$ contains exactly one multi-edge of type $C$.\\
Assume, by contradiction, that every multi-edge of $G_{i_t,i_{t+1}}$ is of type $A$ or $B$. In order to avoid fundamental obstructions that multi-edge incident to $i_t$ and $i_{t+1}$, respectively, has to be of type $A$ and $B$, respectively. In particular, $G_{i_t,i_{t+1}}$ contains at least two consecutive edges. But then, if $i_t\neq i_{t+1}$, there is a partially directed path in clockwise direction in $G_{i_t,i_{t+1}}$. If $i_t=i_{t+1}$, i.e., $G_{i_t,i_{t+1}}=G$, then there is partially directed cylce in clockwise direction in $G$. This yields a contradiction. Since $\vert S \vert = n+a_1+\cdots+a_n$, it follows that $G_{i_t,i_{t+1}}$ contains exactly one multi-edge of type $C$ which establishes the claim.

Note that the previous arguments also show that there always has to be an edge of type $C$, not containing an undirected edge, between the last edge of type $A$ in $G_{i_t,i_{t+1}}$ and the node $i_{t+1}$. Hence, the unique multi-edge of type $C$ can only be followed by multi-edges of type $B$ (or none). In particular, if this edge is incident to $i_t$ in $G_{i_t,i_{t+1}}$, then all remaining multi-edges are of type $B$. Due to the fundamental obstructions, in this case, the edge of type $C$  is not allowed to contain counter-clockwise directed edges, and ,clockwise directed edges are excluded since they  would yield a partially directed path in clockwise direction (or a fundamental obstruction if $i_{t+1}=i_t+1$). Hence, $G_{i_t,i_{t+1}}$ is as described in (2). 

Now assume that the multi-edge of type $C$ of $G_{i_t,i_{t+1}}$ is not incident to $i_t$. The previous arguments show that it can only contain squiggly or counter-clockwise directed edge. We also already know that it has to be followed by multi-edges of type $B$, so it remains to consider the multi-edges between $i_t$ and the multi-edge of type $C$. However, the only restriction we encounter is the multi-edge incident to $i_t$ in $G_{i_t,i_{t+1}}$ being of type $A$. Besides, both type $A$  and $B$ are possible since the multi-edge of type $C$ will always interrupt possible zig-zag or cyclic obstructions. Hence, $G_{i_t,i_{t+1}}$ is as described in (1).
\end{proof}

In order to apply \Cref{thm: visibility formula for h star} to the triangulation $\Ta$ we need to define a point that is in general position with respect to all facets of $\Ta$. 
To this end, define $q \in (0, 1)^{ [n] \cup E(C_{\mathbf{a}})}$ by 
\begin{align*}
    q_i &\coloneqq \frac{1}{n} \cdot \frac{n + \frac{1}{2}}{1 + n } \text{ for all $i \in [n]$,}
\intertext{and}
    q_{e_{j}^{(i)}} &\coloneqq \frac{1}{a_1+\cdots+a_n } \cdot \frac{\frac{1}{2}}{1 + n} \text{ for all $e_j^{(i)}\in E(C_{\mathbf{a}})$.}
\end{align*}
Since $\sum_{v \in [n]} q_v + \sum_{e \in E(C_{\mathbf{a}})} q_e=1$ and $q$ has strictly positive coordinates, we have $q\in \C_{C_{\mathbf{a}}}$. We will see in the proof of the following lemma that $q$ is indeed in general position with respect to any facet of $\Ta$. In order to compute the $h^\ast$-polynomial of $\C_{C_{\mathbf{a}}}$ it remains to determine, how many facets of each facet $S$ of $\Ta$ are visible from $q$. 
To achieve this, for each vertex $p$ of $S$, we will compute the unique homogeneous hyperplane $\mathcal{H}$ containing all vertices of $S$ except $p$ and check if $p$ and $q$ lie on the same or on different sides of $\mathcal{H}$. 
In the former case, the facet not containing $p$ is  \emph{not} visible; in the second case, it is visible. 

\begin{lemma}
    \label{lem: visibility of cycle with multi-edges}
    Let $q \in (0,1)^{[n] \cup E(C_{\mathbf{a}})}$ as defined above. 
    Let $S$ be a facet of $\Ta$, $p$ a vertex of $S$ and let $F$ be the facet of $S$ not containing $p$. As in \Cref{thm: cycle triangulation}, let $Z_S =\{z_{i_1}, \ldots , z_{i_k}\}$, where $i_1 < \cdots < i_k$.
    Then $F$ is visible from $q$ if and only if $p$ is one of the following:
    \begin{enumerate}
    \item a clockwise directed edge contained in a double edge.
    \item counter-clockwise directed edge contained in a double edge that occurs between the unique edge of type $C$ in one of the induced subgraphs $G_{i_t,i_{t+1}}$ and the node $i_{t+1}$.
    \item  an undirected edge contained in a double edge that occurs between $i_t$ and the unique edge of type $C$ in one of the induced subgraphs $G_{i_t,i_{t+1}}$.
   \item a squiggly edge.
   \end{enumerate}
   In all the other cases, $F$ is not visible.
\end{lemma}

\begin{proof}
In the following, we set $G=C_{\mathbf{a}}$ and $E=E(C_{\mathbf{a}})$. 
We need to distinguish different cases according to the type of $p$.\\

\noindent{\sf Case 1 ($p$ is a directed, squiggly or undirected single edge).}
Let $e_j^{(i)}$ be the edge of $C_{\mathbf{a}}$ corresponding to $p$. Define
\[
\mathcal{H}_1=\{w\in \RR^{[n]\cup E}~ :~ w_{e_j^{(i)}}=0\}
\]
and let $\mathcal{H}_1^+$ and $\mathcal{H}_1^-$ be the positive and the negative open halfspace of $\mathcal{H}_1$, respectively. It is immediate that $q \in \mathcal{H}_1^+$, and $p\in\mathcal{H}_1^-$ if and only if $p$ is a squiggly edge. Since $u\in \mathcal{H}_1$ for all vertices of $S$ with $u\neq p$, we conclude that if $p$ is a single edge, $F$ is visible if and only if $p$ is squiggly.  This is case (4) in the lemma.

\noindent{\sf Case 2 ($p$ is a clockwise directed edge contained in a double edge).}
 Let $e_j^{(i)}$ be the edge of $C_{\mathbf{a}}$ corresponding to $p$ and let $G_{i_t,i_{t+1}}$ be the induced subgraph between white nodes containing it. 
 By \Cref{thm: facet description}, $e_j^{(i)}$ is the unique double edge in a multi-edge of type $A$ and $S$ belongs to the facets given by \Cref{thm: cycle triangulation} (1). 
 Locally, starting with the multi-edge between $i$ and $i+1$ and ending with the unique multi-edge of type $C$ in  $G_{i_t,i_{t+1}}$, the induced subgraph looks as follows where gray nodes can be either black or white and $p$ is depicted in blue: 
        
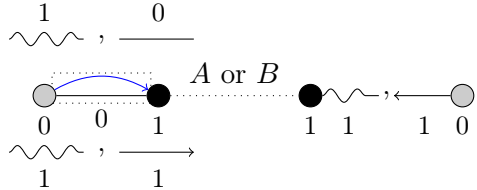
\begin{center}
        \begin{tikzpicture}
            
            \node[black, inner sep=3pt] (1) at (1.5, 0) {};
            \node[] (a) at (0.75,0.75) {,};
            \node[] (b) at (0.75,-0.75) {,};
            \node[] (2) at (-0.6,0.75) {};
            \node[] (3) at (0.65,0.75) {};
            \node[] (4) at (0.85,0.75) {};
            \node[] (5) at (2.1,0.75) {};
            \node[] (6) at (-0.6,-0.75) {};
            \node[] (7) at (0.65,-0.75) {};
            \node[] (8) at (0.85,-0.75) {};
            \node[] (9) at (2.1,-0.75) {};

            \draw[dotted] (0.1,-0.1) rectangle (1.4,0.3);
            \node[gray,inner sep=3pt] (0) at (0, 0) {};
            
            \draw[squiggle] (2) to (3);
            \draw (4) to (5);
            \draw[blue,->] (0) to [bend left] (1);
            \draw (0) to (1);
            \draw[squiggle] (6) to (7);
            \draw[->] (8) to (9);
            
            \node[black,inner sep=3pt] (10) at (3.5, 0) {};
            \node[gray,inner sep=3pt] (11) at (5.5, 0) {};

            \node[] (13) at (2.5,0.3) {$A$ or $B$};
            
            \draw[dotted] (1) to (10);
            
            \draw[squiggle] (10) to ($(10)+(0.9,0)$);
            \draw[<-] ($(10) + (1.1,0)$) to (11);
            \node[] (15) at (4.5,0) {,};

            \node at (0,-0.4) {\small 0};
            \node at (0,1.1) {\small 1};
            \node at (1.5, 1.1) {\small 0};
            \node at (0, -1.1) {\small 1};
            \node at (1.5, -1.1) {\small 1};
            \node at (0.75,-0.3) {\small 0};
            \node at (1.5, -0.4) {\small 1};

            \node at (3.5, -0.4) {\small 1};
            \node at (5.5, -0.4) {\small 0};
            \node at (4, -0.4) {\small 1};
            \node at (5, -0.4) {\small 1};
            
        \end{tikzpicture}
    \end{center}
We now define a vector $c\in \RR^{[n]\cup E}$ as follows: we set $c_k=0$ if $k$ does not occur in the previous picture, i.e., $k$ is external of the two gray nodes depicted above. 
If $k$ is a vertex or an edge depicted above, then we let $c_k$ be the label (if it exists) next to the corresponding vertex or edge in the picture, e.g., $c_f=1$ and $c_f=0$, respectively, if $f$ is a squiggly and undirected edge, respectively, above the double edge $e_j^{(i)}$. 
Finally, if $f$ is an edge of a multi-edge of type $A$ or $B$ without label, then $c_f=2$ if $f$ is a squiggly edge. 
We set $c_f=0$, otherwise. For the non-gray vertices $k$ incident to the edges of type $A$ and type $B$ included in the picture we set $c_k=1$.
We consider the hyperplane
    \begin{align*}
        \mathcal{H}_2 \coloneqq \{w \in \RR^{[n] \cup E} ~:~ c^\top w=0\}.
    \end{align*}
    It is easy to see that $p\in \mathcal{H}_2^-$ and $q\in \mathcal{H}_2^+$. Hence, $F$ is visible from $q$.

\noindent{\sf Case 3 ($p$ is a counter-clockwise directed edge contained in a double edge).}
       If a counterclockwise directed edge is part of a double edge in $G_S$, then again by \Cref{thm: cycle triangulation}, locally the induced subgraph $G_S$ has to fall into one of the following three cases (where $p$ is depicted in blue or red):
    \begin{center}
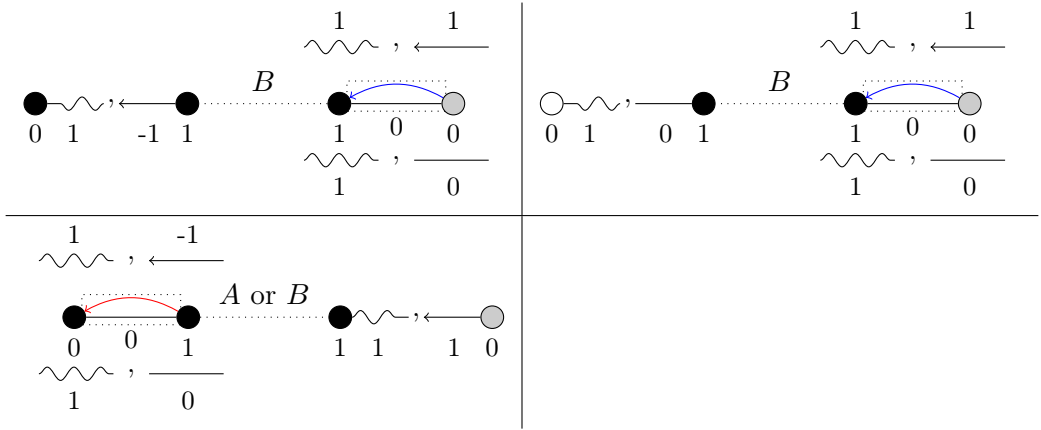
        \begin{tabular}{c|c}
            \begin{tikzpicture}
            \node[] (a) at (0.75,0.75) {,};
            \node[] (b) at (0.75,-0.75) {,};
            \node[] (2) at (-0.6,0.75) {};
            \node[] (3) at (0.65,0.75) {};
            \node[] (4) at (0.85,0.75) {};
            \node[] (5) at (2.1,0.75) {};
            \node[] (6) at (-0.6,-0.75) {};
            \node[] (7) at (0.65,-0.75) {};
            \node[] (8) at (0.85,-0.75) {};
            \node[] (9) at (2.1,-0.75) {};

            \draw[dotted] (0.1,-0.1) rectangle (1.4,0.3);
            \node[black,inner sep=3pt] (0) at (0, 0) {};
            \node[gray, inner sep=3pt] (1) at (1.5, 0) {};
            
            \draw[squiggle] (2) to (3);
            \draw[<-] (4) to (5);
            \draw[blue,<-] (0) to [bend left] (1);
            \draw (0) to (1);
            \draw[squiggle] (6) to (7);
            \draw[-] (8) to (9);
            
            \node[black,inner sep=3pt] (10) at (-2, 0) {};
            \node[black,inner sep=3pt] (11) at (-4, 0) {};

            \node[] (13) at (-1,0.3) {$B$};
            
            \draw[dotted] (0) to (10);
            
            \draw[->] (10) to ($(10)+(-0.9,0)$);
            \draw[squiggle] ($(10) + (-1.1,0)$) to (11);
            \node[] (15) at (-3,0) {,};

            \node at (0,-0.4) {\small 1};
            \node at (0,1.1) {\small 1};
            \node at (1.5, 1.1) {\small 1};
            \node at (0, -1.1) {\small 1};
            \node at (1.5, -1.1) {\small 0};
            \node at (0.75,-0.3) {\small 0};
            \node at (1.5, -0.4) {\small 0};

            \node at (-2, -0.4) {\small 1};
            \node at (-4, -0.4) {\small 0};
            \node at (-2.5, -0.4) {\small -1};
            \node at (-3.5, -0.4) {\small 1};
        \end{tikzpicture}
        &
        \begin{tikzpicture}
            \node[] (a) at (0.75,0.75) {,};
            \node[] (b) at (0.75,-0.75) {,};
            \node[] (2) at (-0.6,0.75) {};
            \node[] (3) at (0.65,0.75) {};
            \node[] (4) at (0.85,0.75) {};
            \node[] (5) at (2.1,0.75) {};
            \node[] (6) at (-0.6,-0.75) {};
            \node[] (7) at (0.65,-0.75) {};
            \node[] (8) at (0.85,-0.75) {};
            \node[] (9) at (2.1,-0.75) {};

            \draw[dotted] (0.1,-0.1) rectangle (1.4,0.3);
            \node[black,inner sep=3pt] (0) at (0, 0) {};
            \node[gray, inner sep=3pt] (1) at (1.5, 0) {};
            
            \draw[squiggle] (2) to (3);
            \draw[<-] (4) to (5);
            \draw[blue,<-] (0) to [bend left] (1);
            \draw (0) to (1);
            \draw[squiggle] (6) to (7);
            \draw[-] (8) to (9);
            
            \node[black,inner sep=3pt] (10) at (-2, 0) {};
            \node[white,inner sep=3pt] (11) at (-4, 0) {};

            \node[] (13) at (-1,0.3) {$B$};
            
            \draw[dotted] (0) to (10);
            
            \draw (10) to ($(10)+(-0.9,0)$);
            \draw[squiggle] ($(10) + (-1.1,0)$) to (11);
            \node[] (15) at (-3,0) {,};

            \node at (0,-0.4) {\small 1};
            \node at (0,1.1) {\small 1};
            \node at (1.5, 1.1) {\small 1};
            \node at (0, -1.1) {\small 1};
            \node at (1.5, -1.1) {\small 0};
            \node at (0.75,-0.3) {\small 0};
            \node at (1.5, -0.4) {\small 0};

            \node at (-2, -0.4) {\small 1};
            \node at (-4, -0.4) {\small 0};
            \node at (-2.5, -0.4) {\small 0};
            \node at (-3.5, -0.4) {\small 1};        \end{tikzpicture} \\ \hline
        \begin{tikzpicture}
            
            \node[black, inner sep=3pt] (1) at (1.5, 0) {};
            \node[] (a) at (0.75,0.75) {,};
            \node[] (b) at (0.75,-0.75) {,};
            \node[] (2) at (-0.6,0.75) {};
            \node[] (3) at (0.65,0.75) {};
            \node[] (4) at (0.85,0.75) {};
            \node[] (5) at (2.1,0.75) {};
            \node[] (6) at (-0.6,-0.75) {};
            \node[] (7) at (0.65,-0.75) {};
            \node[] (8) at (0.85,-0.75) {};
            \node[] (9) at (2.1,-0.75) {};

            \draw[dotted] (0.1,-0.1) rectangle (1.4,0.3);
            \node[black,inner sep=3pt] (0) at (0, 0) {};
            
            \draw[squiggle] (2) to (3);
            \draw[<-] (4) to (5);
            \draw[red,<-] (0) to [bend left] (1);
            \draw (0) to (1);
            \draw[squiggle] (6) to (7);
            \draw (8) to (9);
            
            \node[black,inner sep=3pt] (10) at (3.5, 0) {};
            \node[gray,inner sep=3pt] (11) at (5.5, 0) {};

            \node[] (13) at (2.5,0.3) {$A$ or $B$};
            
            \draw[dotted] (1) to (10);
            
            \draw[squiggle] (10) to ($(10)+(0.9,0)$);
            \draw[<-] ($(10) + (1.1,0)$) to (11);
            \node[] (15) at (4.5,0) {,};

            \node at (0,-0.4) {\small 0};
            \node at (0,1.1) {\small 1};
            \node at (1.5, 1.1) {\small -1};
            \node at (0, -1.1) {\small 1};
            \node at (1.5, -1.1) {\small 0};
            \node at (0.75,-0.3) {\small 0};
            \node at (1.5, -0.4) {\small 1};

            \node at (3.5, -0.4) {\small 1};
            \node at (5.5, -0.4) {\small 0};
            \node at (4, -0.4) {\small 1};
            \node at (5, -0.4) {\small 1}; 
        \end{tikzpicture}
        \end{tabular}
    \end{center}
  We define a hyperplane $\mathcal{H}_3$ in the same way as we did in Case $2$, using the labels from the pictures. Note that also for depicted multi-edges without labels we use the same rules. It is easy to verify that in the first two cases, $p$ and $q$ lie in different halfspaces of  $\mathcal{H}_3$ and hence, $F$ is visible whereas in the third case they lie in the same halfspace and $F$ is thus not visible. 

\noindent{\sf Case 4 ($p$ is an undirected edge contained in a double edge).}
   As before, using \Cref{thm: cycle triangulation} for the local picture we have the following possibilities:
    \begin{center}
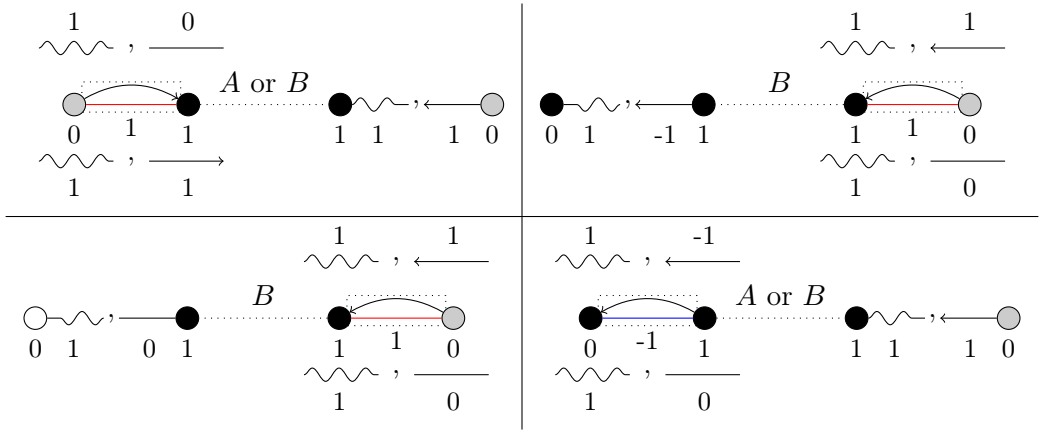
        \begin{tabular}{c|c}
        \begin{tikzpicture}
            \node[black, inner sep=3pt] (1) at (1.5, 0) {};
            \node[] (a) at (0.75,0.75) {,};
            \node[] (b) at (0.75,-0.75) {,};
            \node[] (2) at (-0.6,0.75) {};
            \node[] (3) at (0.65,0.75) {};
            \node[] (4) at (0.85,0.75) {};
            \node[] (5) at (2.1,0.75) {};
            \node[] (6) at (-0.6,-0.75) {};
            \node[] (7) at (0.65,-0.75) {};
            \node[] (8) at (0.85,-0.75) {};
            \node[] (9) at (2.1,-0.75) {};

            \draw[dotted] (0.1,-0.1) rectangle (1.4,0.3);
            \node[gray,inner sep=3pt] (0) at (0, 0) {};
            
            \draw[squiggle] (2) to (3);
            \draw (4) to (5);
            \draw[->] (0) to [bend left] (1);
            \draw[red] (0) to (1);
            \draw[squiggle] (6) to (7);
            \draw[->] (8) to (9);
            
            \node[black,inner sep=3pt] (10) at (3.5, 0) {};
            \node[gray,inner sep=3pt] (11) at (5.5, 0) {};

            \node[] (13) at (2.5,0.3) {$A$ or $B$};
            
            \draw[dotted] (1) to (10);
            
            \draw[squiggle] (10) to ($(10)+(0.9,0)$);
            \draw[<-] ($(10) + (1.1,0)$) to (11);
            \node[] (15) at (4.5,0) {,};

            \node at (0,-0.4) {\small 0};
            \node at (0,1.1) {\small 1};
            \node at (1.5, 1.1) {\small 0};
            \node at (0, -1.1) {\small 1};
            \node at (1.5, -1.1) {\small 1};
            \node at (0.75,-0.3) {\small 1};
            \node at (1.5, -0.4) {\small 1};

            \node at (3.5, -0.4) {\small 1};
            \node at (5.5, -0.4) {\small 0};
            \node at (4, -0.4) {\small 1};
            \node at (5, -0.4) {\small 1};
        \end{tikzpicture} &
        \begin{tikzpicture}
        
            \node[] (a) at (0.75,0.75) {,};
            \node[] (b) at (0.75,-0.75) {,};
            \node[] (2) at (-0.6,0.75) {};
            \node[] (3) at (0.65,0.75) {};
            \node[] (4) at (0.85,0.75) {};
            \node[] (5) at (2.1,0.75) {};
            \node[] (6) at (-0.6,-0.75) {};
            \node[] (7) at (0.65,-0.75) {};
            \node[] (8) at (0.85,-0.75) {};
            \node[] (9) at (2.1,-0.75) {};

            \draw[dotted] (0.1,-0.1) rectangle (1.4,0.3);
            \node[black,inner sep=3pt] (0) at (0, 0) {};
            \node[gray, inner sep=3pt] (1) at (1.5, 0) {};
            
            \draw[squiggle] (2) to (3);
            \draw[<-] (4) to (5);
            \draw[<-] (0) to [bend left] (1);
            \draw[red] (0) to (1);
            \draw[squiggle] (6) to (7);
            \draw[-] (8) to (9);
            
            \node[black,inner sep=3pt] (10) at (-2, 0) {};
            \node[black,inner sep=3pt] (11) at (-4, 0) {};

            \node[] (13) at (-1,0.3) {$B$};
            
            \draw[dotted] (0) to (10);
            
            \draw[->] (10) to ($(10)+(-0.9,0)$);
            \draw[squiggle] ($(10) + (-1.1,0)$) to (11);
            \node[] (15) at (-3,0) {,};

            \node at (0,-0.4) {\small 1};
            \node at (0,1.1) {\small 1};
            \node at (1.5, 1.1) {\small 1};
            \node at (0.75,-0.3) {\small 1};
            \node at (0, -1.1) {\small 1};
            \node at (1.5, -1.1) {\small 0};
            \node at (1.5, -0.4) {\small 0};

            \node at (-2, -0.4) {\small 1};
            \node at (-2.5, -0.4) {\small -1};
            \node at (-3.5, -0.4) {\small 1};
            \node at (-4, -0.4) {\small 0};
        \end{tikzpicture}
        \\ \hline
        \begin{tikzpicture}

            \node[] (a) at (0.75,0.75) {,};
            \node[] (b) at (0.75,-0.75) {,};
            \node[] (2) at (-0.6,0.75) {};
            \node[] (3) at (0.65,0.75) {};
            \node[] (4) at (0.85,0.75) {};
            \node[] (5) at (2.1,0.75) {};
            \node[] (6) at (-0.6,-0.75) {};
            \node[] (7) at (0.65,-0.75) {};
            \node[] (8) at (0.85,-0.75) {};
            \node[] (9) at (2.1,-0.75) {};

            \draw[dotted] (0.1,-0.1) rectangle (1.4,0.3);
            \node[black,inner sep=3pt] (0) at (0, 0) {};
            \node[gray, inner sep=3pt] (1) at (1.5, 0) {};
            
            \draw[squiggle] (2) to (3);
            \draw[<-] (4) to (5);
            \draw[<-] (0) to [bend left] (1);
            \draw[red] (0) to (1);
            \draw[squiggle] (6) to (7);
            \draw[-] (8) to (9);
            
            \node[black,inner sep=3pt] (10) at (-2, 0) {};
            \node[white,inner sep=3pt] (11) at (-4, 0) {};

            \node[] (13) at (-1,0.3) {$B$};
            
            \draw[dotted] (0) to (10);
            
            \draw (10) to ($(10)+(-0.9,0)$);
            \draw[squiggle] ($(10) + (-1.1,0)$) to (11);
            \node[] (15) at (-3,0) {,};

            \node at (0,-0.4) {\small 1};
            \node at (0,1.1) {\small 1};
            \node at (1.5, 1.1) {\small 1};
            \node at (0.75,-0.3) {\small 1};
            \node at (0, -1.1) {\small 1};
            \node at (1.5, -1.1) {\small 0};
            \node at (1.5, -0.4) {\small 0};

            \node at (-2, -0.4) {\small 1};
            \node at (-2.5, -0.4) {\small 0};
            \node at (-3.5, -0.4) {\small 1};
            \node at (-4, -0.4) {\small 0};
        \end{tikzpicture} &
        \begin{tikzpicture}
            \node[black, inner sep=3pt] (1) at (1.5, 0) {};
            \node[] (a) at (0.75,0.75) {,};
            \node[] (b) at (0.75,-0.75) {,};
            \node[] (2) at (-0.6,0.75) {};
            \node[] (3) at (0.65,0.75) {};
            \node[] (4) at (0.85,0.75) {};
            \node[] (5) at (2.1,0.75) {};
            \node[] (6) at (-0.6,-0.75) {};
            \node[] (7) at (0.65,-0.75) {};
            \node[] (8) at (0.85,-0.75) {};
            \node[] (9) at (2.1,-0.75) {};

            \draw[dotted] (0.1,-0.1) rectangle (1.4,0.3);
            \node[black,inner sep=3pt] (0) at (0, 0) {};
            
            \draw[squiggle] (2) to (3);
            \draw[<-] (4) to (5);
            \draw[<-] (0) to [bend left] (1);
            \draw[blue] (0) to (1);
            \draw[squiggle] (6) to (7);
            \draw (8) to (9);
            
            \node[black,inner sep=3pt] (10) at (3.5, 0) {};
            \node[gray,inner sep=3pt] (11) at (5.5, 0) {};

            \node[] (13) at (2.5,0.3) {$A$ or $B$};
            
            \draw[dotted] (1) to (10);
            
            \draw[squiggle] (10) to ($(10)+(0.9,0)$);
            \draw[<-] ($(10) + (1.1,0)$) to (11);
            \node[] (15) at (4.5,0) {,};

            \node at (0,-0.4) {\small 0};
            \node at (0,1.1) {\small 1};
            \node at (1.5, 1.1) {\small -1};
            \node at (0, -1.1) {\small 1};
            \node at (1.5, -1.1) {\small 0};
            \node at (0.75,-0.3) {\small -1};
            \node at (1.5, -0.4) {\small 1};

            \node at (3.5, -0.4) {\small 1};
            \node at (5.5, -0.4) {\small 0};
            \node at (4, -0.4) {\small 1};
            \node at (5, -0.4) {\small 1};
        \end{tikzpicture}
    \end{tabular}
    \end{center}
     We define a hyperplane $\mathcal{H}_4$ in the same way as we did in Case $2$, using the labels from the pictures. Note that also for depicted multi-edges without labels we use the same rules. It is easy to verify that in the first three cases, $p$ and $q$ lie in the same halfspace and $F$ is hence not visible whereas in the fourth case, $p$ and $q$ lie in different halfspaces and thus $F$ is visible.

    \noindent{\sf Case 5 ($p$ is a white node).}
    Suppose first that $p$ is the only one white node in $G_S$. In this case, we set $c_v=1$ for all $v \in [n]$, $c_e=2$ if $e$ appears as a squiggly edge in $G_S$ and $c_e=0$, otherwise. Setting
    \begin{align*}
        \mathcal{H}_5 \coloneqq \{w \in \RR^{[n]  \cup E} ~:~w^\top c=0\},
    \end{align*}
    we see that $p$ and $q$ both lie in $\mathcal{H}_5^+$, which shows that $F$ is not visible from $q$.

    Suppose that $G_S$ contains at least two white nodes. Using \Cref{thm: cycle triangulation}, it follows, that going from $v$ in counterclockwise direction, one of the following four cases occurs (where $p$ is depicted in red):
    \begin{center}
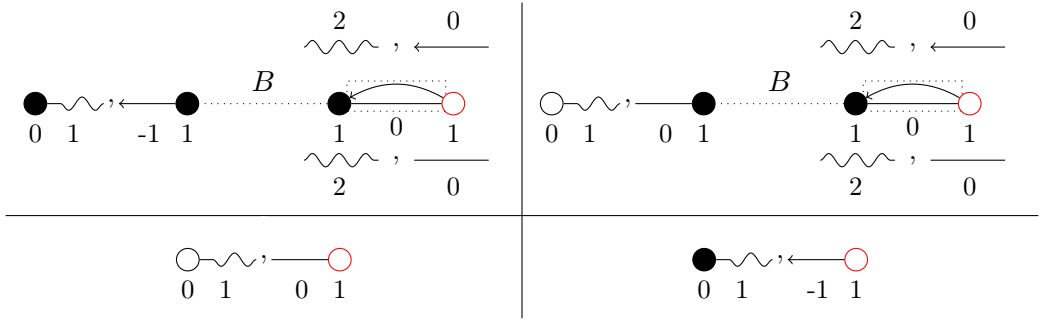
        \begin{tabular}{c|c}
            \begin{tikzpicture}
            \node[] (a) at (0.75,0.75) {,};
            \node[] (b) at (0.75,-0.75) {,};
            \node[] (2) at (-0.6,0.75) {};
            \node[] (3) at (0.65,0.75) {};
            \node[] (4) at (0.85,0.75) {};
            \node[] (5) at (2.1,0.75) {};
            \node[] (6) at (-0.6,-0.75) {};
            \node[] (7) at (0.65,-0.75) {};
            \node[] (8) at (0.85,-0.75) {};
            \node[] (9) at (2.1,-0.75) {};

            \draw[dotted] (0.1,-0.1) rectangle (1.4,0.3);
            \node[black,inner sep=3pt] (0) at (0, 0) {};
            \node[white, draw=red, inner sep=3pt] (1) at (1.5, 0) {};
            
            \draw[squiggle] (2) to (3);
            \draw[<-] (4) to (5);
            \draw[<-] (0) to [bend left] (1);
            \draw (0) to (1);
            \draw[squiggle] (6) to (7);
            \draw[-] (8) to (9);
            
            \node[black,inner sep=3pt] (10) at (-2, 0) {};
            \node[black,inner sep=3pt] (11) at (-4, 0) {};

            \node[] (13) at (-1,0.3) {$B$};
            
            \draw[dotted] (0) to (10);
            
            \draw[->] (10) to ($(10)+(-0.9,0)$);
            \draw[squiggle] ($(10) + (-1.1,0)$) to (11);
            \node[] (15) at (-3,0) {,};

            \node at (0,-0.4) {\small 1};
            \node at (0,1.1) {\small 2};
            \node at (1.5, 1.1) {\small 0};
            \node at (0.75,-0.3) {\small 0};
            \node at (0, -1.1) {\small 2};
            \node at (1.5, -1.1) {\small 0};
            \node at (1.5, -0.4) {\small 1};

            \node at (-2, -0.4) {\small 1};
            \node at (-2.5, -0.4) {\small -1};
            \node at (-3.5, -0.4) {\small 1};
            \node at (-4, -0.4) {\small 0};
            \end{tikzpicture}
            &
            \begin{tikzpicture}
                \node[] (a) at (0.75,0.75) {,};
            \node[] (b) at (0.75,-0.75) {,};
            \node[] (2) at (-0.6,0.75) {};
            \node[] (3) at (0.65,0.75) {};
            \node[] (4) at (0.85,0.75) {};
            \node[] (5) at (2.1,0.75) {};
            \node[] (6) at (-0.6,-0.75) {};
            \node[] (7) at (0.65,-0.75) {};
            \node[] (8) at (0.85,-0.75) {};
            \node[] (9) at (2.1,-0.75) {};

            \draw[dotted] (0.1,-0.1) rectangle (1.4,0.3);
            \node[black,inner sep=3pt] (0) at (0, 0) {};
            \node[white, draw=red, inner sep=3pt] (1) at (1.5, 0) {};
            
            \draw[squiggle] (2) to (3);
            \draw[<-] (4) to (5);
            \draw[<-] (0) to [bend left] (1);
            \draw (0) to (1);
            \draw[squiggle] (6) to (7);
            \draw[-] (8) to (9);
            
            \node[black,inner sep=3pt] (10) at (-2, 0) {};
            \node[white,inner sep=3pt] (11) at (-4, 0) {};

            \node[] (13) at (-1,0.3) {$B$};
            
            \draw[dotted] (0) to (10);
            
            \draw (10) to ($(10)+(-0.9,0)$);
            \draw[squiggle] ($(10) + (-1.1,0)$) to (11);
            \node[] (15) at (-3,0) {,};

            \node at (0,-0.4) {\small 1};
            \node at (0,1.1) {\small 2};
            \node at (1.5, 1.1) {\small 0};
            \node at (0.75,-0.3) {\small 0};
            \node at (0, -1.1) {\small 2};
            \node at (1.5, -1.1) {\small 0};
            \node at (1.5, -0.4) {\small 1};

            \node at (-2, -0.4) {\small 1};
            \node at (-2.5, -0.4) {\small 0};
            \node at (-3.5, -0.4) {\small 1};
            \node at (-4, -0.4) {\small 0};
            \end{tikzpicture} \\ \hline
            
            \begin{tikzpicture}
            
            \node[white, draw=red,inner sep=3pt] (10) at (-2, 0) {};
            \node[white,inner sep=3pt] (11) at (-4, 0) {};
            
            \draw (10) to ($(10)+(-0.9,0)$);
            \draw[squiggle] ($(10) + (-1.1,0)$) to (11);
            \node[] (15) at (-3,0) {,};

            \draw[white, draw=white] (-3,0.5) circle (2pt);

            \node at (-2, -0.4) {\small 1};
            \node at (-2.5, -0.4) {\small 0};
            \node at (-3.5, -0.4) {\small 1};
            \node at (-4, -0.4) {\small 0};
            \end{tikzpicture}
            &
            \begin{tikzpicture}
            \node[white, draw=red,inner sep=3pt] (10) at (-2, 0) {};
            \node[black,inner sep=3pt] (11) at (-4, 0) {};
            
            \draw[->] (10) to ($(10)+(-0.9,0)$);
            \draw[squiggle] ($(10) + (-1.1,0)$) to (11);
            \node (15) at (-3,0) {,};

            \node at (-2, -0.4) {\small 1};
            \node at (-2.5, -0.4) {\small -1};
            \node at (-3.5, -0.4) {\small 1};
            \node at (-4, -0.4) {\small 0};
            \end{tikzpicture}
        \end{tabular}
    \end{center}
    
    Similarly, going from $v$ in clockwise direction, the local picture is one of the following two (where $p$ is depicted in red):
    \begin{center}
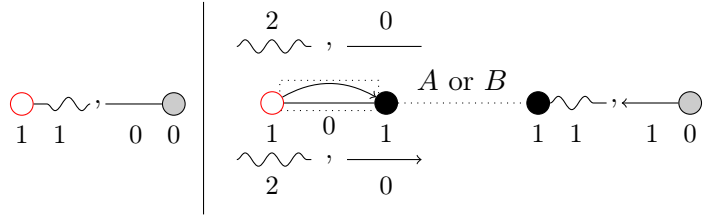
        \begin{tabular}{c|c}
            \begin{tikzpicture}
            \draw[white,draw=white] (-3,-1.25) circle (2pt);
            
            \node[gray,inner sep=3pt] (10) at (-2, 0) {};
            \node[white,draw=red,inner sep=3pt] (11) at (-4, 0) {};
            
            \draw (10) to ($(10)+(-0.9,0)$);
            \draw[squiggle] ($(10) + (-1.1,0)$) to (11);
            \node[] (15) at (-3,0) {,};

            \node at (-2, -0.4) {\small 0};
            \node at (-2.5, -0.4) {\small 0};
            \node at (-3.5, -0.4) {\small 1};
            \node at (-4, -0.4) {\small 1};
            \end{tikzpicture}
            &
            \begin{tikzpicture}
            
            \node[black, inner sep=3pt] (1) at (1.5, 0) {};
            \node[] (a) at (0.75,0.75) {,};
            \node[] (b) at (0.75,-0.75) {,};
            \node[] (2) at (-0.6,0.75) {};
            \node[] (3) at (0.65,0.75) {};
            \node[] (4) at (0.85,0.75) {};
            \node[] (5) at (2.1,0.75) {};
            \node[] (6) at (-0.6,-0.75) {};
            \node[] (7) at (0.65,-0.75) {};
            \node[] (8) at (0.85,-0.75) {};
            \node[] (9) at (2.1,-0.75) {};

            \draw[dotted] (0.1,-0.1) rectangle (1.4,0.3);
            \node[white,draw=red,inner sep=3pt] (0) at (0, 0) {};
            
            \draw[squiggle] (2) to (3);
            \draw (4) to (5);
            \draw[->] (0) to [bend left] (1);
            \draw (0) to (1);
            \draw[squiggle] (6) to (7);
            \draw[->] (8) to (9);
            
            \node[black,inner sep=3pt] (10) at (3.5, 0) {};
            \node[gray,inner sep=3pt] (11) at (5.5, 0) {};

            \node[] (13) at (2.5,0.3) {$A$ or $B$};
            
            \draw[dotted] (1) to (10);
            
            \draw[squiggle] (10) to ($(10)+(0.9,0)$);
            \draw[<-] ($(10) + (1.1,0)$) to (11);
            \node[] (15) at (4.5,0) {,};

            \node at (0,-0.4) {\small 1};
            \node at (0,1.1) {\small 2};
            \node at (1.5, 1.1) {\small 0};
            \node at (0.75,-0.3) {\small 0};
            \node at (0, -1.1) {\small 2};
            \node at (1.5, -1.1) {\small 0};
            \node at (1.5, -0.4) {\small 1};

            \node at (3.5, -0.4) {\small 1};
            \node at (5.5, -0.4) {\small 0};
            \node at (4, -0.4) {\small 1};
            \node at (5, -0.4) {\small 1};
        \end{tikzpicture}
        \end{tabular}
        \end{center} 
        In particular, we get eight different local situations by combining a picture from above and below. In each of these cases, we define a hyperplane in the same way as we did before, using the labels from the combined pictures. Again depicted multi-edges without labels are treated as in Case 2 (and all other cases). It is easy to verify that in all cases, $p$ and $q$ lie on the same side of this hyperplane and hence, $F$ is never visible from $q$.

As $q$ was not contained in any of the defined hyperplanes, it is indeed a point in general position with respect to any facet of $\Ta$.
\end{proof}

As a direct consequence of \Cref{lem: visibility of cycle with multi-edges} we get the following simple expression for the number of visible facets of a facet $S$ of $\Ta$:

\begin{corollary}
    \label{cor: visibility characterization}
    Let $S$ be a facet of $\Ta$ and let $q \in \RR^{[n]\cup E(C_{\mathbf{a}})}$ as defined above. 
    Then the number of facets of $S$ that are visible from $q$ is equal to the number of edges in $G_S$ that are squiggly edges or double edges. 
\end{corollary}

The next lemma shows that a facet of $\Tc_\textbf{a}$ is uniquely determined by its oriented double edges and its squiggly edges.

\begin{lemma}\label{lem: edge arrangements}
    Let $S$ be a facet of $\Ta$ and let $S'\subseteq S$ be the set of the squiggly and double edges in $S$.
    Then $S$ is the unique facet of $\Ta$ whose set of squiggly and double edges equals $S'$.
Moreover, any set of double edges and squiggle edges such that each multi-edge contains at most one double edge, can be extended to a facet of $\Ta$, if and only if there is at least one multi-edge without any double edge.
\end{lemma}

\begin{proof}
We show that $S$ can be uniquely reconstructed from $S'$. We first note that by \Cref{lem: induced G_S on one multi edge} the double edges uniquely determine if a multi-edge is of type $A$ or $B$. Knowing the squiggly edges in such a multi-edge determines the other single edges uniquely. It hence remains to reconstruct edges of type $C$ and white nodes. 

Since we know all multi-edges of type $A$ and $B$, also the location of the multi-edges of type $C$ are determined. Starting at such a multi-edge of type $C$ and going in clockwise direction, we traverse the cycle until we see a multi-edge of type $A$ or another multi-edge of type $C$ for the first time. The starting node of this multi-edge (in clockwise direction) has to be a white node by \Cref{thm: cycle triangulation}. In this way, white nodes are uniquely determined. Since, finally, it is now also determined if each induced subgraph falls into the facets in \Cref{thm: cycle triangulation} (1) or (2), the edges of type $C$ are uniquely determined. This shows the first claim.

The "moreover" follows from our construction and the fact that otherwise, $S'$ would contain the leading term of a cyclic binomial.
\end{proof}

Combining \Cref{cor: visibility characterization} and \Cref{lem: edge arrangements}, we may explicitly compute the $h^\ast$-polynomial of $\C_{C_{\mathbf{a}}}$.

\begin{theorem}
    \label{thm: h-star of cycle with multiedges}
    Let $\mathbf{a}=(a_1,\ldots,a_n)\in \mathbb{Z}^n$ with $a_i\geq 1$ for all $1\leq i\leq n$. The $h^\ast$-polynomial of $\C_{C_{\mathbf{a}}}$ is given by
    \[
        h^\ast(\C_{C_{\mathbf{a}}}; z) = \prod_{i=1}^n ((1 + z)^{a_i} + 2a_iz(1+z)^{a_i - 1}) - \prod_{i=1}^n 2a_iz(1+z)^{a_i - 1}.
    \]
\end{theorem}
\begin{proof}
    On the one hand, by \Cref{lem: edge arrangements}, a facet $S$ of $\Ta$ is uniquely determined by its set of double edges and squiggly edges. On the other hand, the cardinality of the latter set is equal to the number of visible facets of $S$ by \Cref{cor: visibility characterization}. To compute $h_k^\ast(\C_{C_{\mathbf{a}}})$ we hence need to count the number of facets of $\Ta$ having $k$ edges that are double or squiggly by \Cref{thm: visibility formula for h star}.

    Instead of counting these facets directly, we first count a slightly larger set. Namely, all sets $S'$ consisting of exactly $k$ edges that are double or squiggly such that for each multi-edge at most one double edge occurs. We then need to subtract the number of those that cannot be completed to a facet of $\Ta$.

First note that variables in such a set $S'$ corresponding to different multi-edges of $C_\textbf{a}$ can be chosen independently. For $i \in [n]$ and $0\leq k\leq a_i$, we let $v_{i,k}$ denote the number of possibilities to choose $k$ edges of the multi-edge between $i$ and $i+1$ that are squiggly or double with at most one double edge. If none of the edges of the multi-edge between $i$ and $i+1$ is a double, there are $\binom{a_i}{k}$ possibilities to choose $k$ edges of this multi-edge that are squiggly. Similarly, there are $2a_i \binom{a_i-1}{k-1}$ possibilities to choose one double edge (for which we have two possible orientations) and $k-1$ squiggle edges from the multi-edge between $i$ and $i+1$. This yields:
\[
v_{i,k}=\binom{a_i}{k}+2a_i\binom{a_i-1}{k-1},
\]
where we set $\binom{a_i-1}{-1} \coloneqq 0$. We now define
    \begin{align*}
        v_i(z) \coloneqq& \sum_{k=0}^{a_i}v_{i,k}z^k=\sum_{k=0}^{a_i}\left(\binom{a_i}{k}+2a_i\binom{a_i-1}{k-1} \right)z^k \\
        =& \sum_{k=0}^{a_i}\binom{a_i}{k}z^k+2a_i\sum_{k=0}^{a_i-1}\binom{a_i-1}{k}z^{k+1} \\
        =&(1+z)^{a_i}+2a_iz(1+z)^{a_i-1}.
    \end{align*}
    Hence, the coefficient of $z^k$ in the polynomial
    \begin{align*}
        v(z) \coloneqq \prod_{i=1}^n v_i(z)
    \end{align*}
    counts the number of possible sets $S'$ consisting of exactly $k$ edges that are double or squiggly such that for each multi-edge at most one double edge occurs.
    By \Cref{lem: edge arrangements} these are those sets $S'$ with one double edge for each multi-edge of $C_\textbf{a}$. As argued above, there are $2a_i\binom{a_i-1}{k-1}$ possibilities that one edge of the multi-edge between $i$ and $i+1$ is a double edge and $k-1$ edges of this multi-edge are squiggly edges and this is counted by the coefficient of $z^k$ in the polynomial $2a_iz(1+z)^{a_i-1}$. For the whole graph this means the coefficient of $z^k$ in the polynomial
    \begin{align*}
        \prod_{i=1}^{n} 2a_iz(1+z)^{a_i-1}
    \end{align*}
    counts the number of sets $S'$ of cardinality $k$ that contain a double edge for each multi-edge and $k-n$ squiggly edges. We finally conclude that  the number of facets of $\Ta$ with exactly $k$ visible facets is given by the coefficient of $z^k$ of the polynomial
    \begin{align*}
        \prod_{i=1}^n ((1+z)^{a_i}+2a_iz(1+z)^{a_i - 1}) - \prod_{i=1}^n 2a_iz(1+z)^{a_i-1},
    \end{align*}
    and by \Cref{thm: visibility formula for h star} this is exactly the $h^\ast$-polynomial of $\C_{C_{\mathbf{a}}}$.
\end{proof}

Since the normalized volume $\Vol(P)$ of a lattice polytope $P$ is given by $h^\ast(P;1)$, we directly get the following.

\begin{corollary}
	 Let $\mathbf{a}=(a_1,\ldots,a_n)\in \mathbb{Z}^n$ with $a_i\geq 1$ for all $1\leq i\leq n$. The normalized volume of $\C_{C_{\mathbf{a}}}$ is given by
	\[
		\Vol(\C_{C_{\mathbf{a}}})=\prod_{i=1}^n(1+a_i)2^{a_i}-\prod_{i=1}^n a_i2^{a_i}.
	\]
\end{corollary}

\begin{example}
    Let $C$ be a simple cycle of length $n$. Using \Cref{thm: h-star of cycle with multiedges}, we get
    \begin{align*}
        h^\ast(\C_{C};z)=(1+3z)^n-(2z)^n \text{ and } \Vol(\C_{C})=4^n-2^n.
    \end{align*}
    This generalizes \cite[Theorem 4.2]{jsl-cosmo} where the normalized volume of $\C_C$ was computed.
\end{example}

\subsection{Multitrees}
\label{subsec: tree polynomial}

\begin{figure}
    \centering
    \begin{tikzpicture}[scale=0.9]
        \coordinate (1) at (0,0);
        \coordinate (2) at (1,0);
        \coordinate (3) at (1,1);
        \coordinate (4) at (1,-1);
        \coordinate (5) at (2,0);
        \coordinate (6) at (135:1);
        \coordinate (7) at (225:1);

        \draw (7) -- (1) -- (2) -- (3);
        \draw (5) -- (2) -- (4);

        \path [bend left=70] (1) edge (2);
        \path[bend left] (1) edge (2);
        \path [bend right=70] (1) edge (2);
        \path[bend right] (1) edge (2);
        \path [bend left] (1) edge (6);
        \path [bend left] (1) edge (7);
        \path [bend right] (1) edge (6);
        \path [bend right] (1) edge (7);
        \path [bend left] (2) edge (3);
        \path [bend right] (2) edge (3);
        \path [bend left] (2) edge (5);
        \path [bend right] (2) edge (5);

        \fill (1) circle (2pt);
        \fill (2) circle (2pt);
        \fill (3) circle (2pt);
        \fill (4) circle (2pt);
        \fill (5) circle (2pt);
        \fill (6) circle (2pt);
        \fill (7) circle (2pt);

    \end{tikzpicture}
    \caption{Example for a multitree of type $(1,2,3,3,3,5)$.}
    \label{fig: example multitree}
\end{figure}
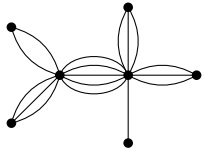

Given a positive integer $n$ and $\mathbf{a}\in\mathbb{Z}^n$ with $a_i\geq 1$ for all $1\leq i\leq n$. A \emph{multitree} of type $\mathbf{a}$ is a connected graph on vertex set $[n+1]$ whose edge set consists of $n$ multi-edges $F_1,\ldots,F_{n}$ where $F_i$ has multiplicity $a_i$, see \Cref{fig: example multitree} for an example. Alternatively a multitree of type $\mathbf{a}$ can be obtained as $1$-sum of multi-edges of multiplicity $a_1,\ldots,a_n$. As such, \Cref{thm: 1-sum} implies that the $h^\ast$-polynomial of the cosmological polytope of any multitree $T$ of type $\mathbf{a}$ does only depend on $\mathbf{a}$. More precisely,
\[
h^\ast(\C_T;z)=\prod_{i=1}^nh^\ast(\C_{I_{a_i}};z),
\]
where $I_{a_i}$ denotes the graph on vertex set $\{1,2\}$ and a multi-edge of multiplicity $a_i$ between $1$ and $2$.
It hence suffices to compute  the $h^\ast$-polynomial of a multi-edge. 

We now state the main result of this section:

\begin{theorem}\label{thm:Ehrhart multitree}
Let $n$ be a positive integer and let $\mathbf{a}\in\mathbb{Z}^n$ with $a_i\geq 1$ for all $1\leq i\leq n$. Let $T$ be a multitree of type $\mathbf{a}$. Then
    \begin{equation*}
        h^\ast(\C_{T};z) = \prod_{i=1}^n \left((1+z)^{a_i} +2a_iz(1+z)^{a_i-1}\right).
    \end{equation*}
\end{theorem}

\begin{remark}
If $T$ is a multi-forest, i.e., the connected components of $T$ are multitrees of a certain types, then  \Cref{prop:union} implies that the $h^\ast$-polynomial of the cosmological polytope of $T$ is the product of the $h^\ast$-polynomials of the components and in particular, \Cref{thm:Ehrhart multitree} holds for multi-forests.
\end{remark}

As argued before \Cref{thm:Ehrhart multitree} will be a direct consequence of the following statement:

\begin{lemma}\label{lem: h*1-path}
Let $m$ be a positive integer.  The $h^\ast$-polynomial of $\mathcal{C}_{I_m}$ is 
    \begin{equation*}
        h^\ast(\mathcal{C}_{I_m};z) = (1+z)^m +2mz(1+z)^{m-1}.
    \end{equation*}
\end{lemma}

To prove \Cref{lem: h*1-path} we will use the same methods as in \Cref{subsec: cycles} for multicycles. Given a multi-edge $I_m$ of multiplicity $m$, we denote the edges of this multi-edge by $e_1,\ldots,e_m$. We first need to define a good term order on the polynomial ring $R_m=R_{\C_{I_m}}$. 
To keep notation as simple as possible, given an edge $e_i$ ($1\leq i\leq m$) we write
\begin{itemize}
\item[$\bullet$] $\overrightarrow{y}_{i}$ and $\overleftarrow{y}_{i}$ for the variables $\overrightarrow{y}_{12e_i}$ and $\overleftarrow{y}_{12e_i}$, respectively,
\item[$\bullet$] $t_i$ for the variable $t_{12e_i}$, and as usual $z_{e_i}$ and $z_1$ and $z_2$.
\end{itemize}
We let $<$ be the lexicographic term order on $R_m$ induced by
\begin{align*}
   & \overrightarrow{y}_{1}>\overrightarrow{y}_{2}>\cdots >\overrightarrow{y}_{m} 
   >\overleftarrow{y}_{m}>\cdots > \overleftarrow{y}_{1}\\
   >&z_{e_1}>\cdots >z_{e_m}>t_{1}>\cdots >t_{m}>z_1>z_2.
\end{align*}
We denote by $\Tm$ the unimodular triangulation from \Cref{cor: unimodular triangulation} corresponding to the term order just defined. 
We observe that the term order, we used for multicycles, restricts to the above term order, when restricted to the polynomial ring on the variables corresponding to a particular multi-edge of cardinality $m$. In particular, \Cref{lem: induced G_S on one multi edge} can be applied. This means that given a facet $S$ of $\Tm$ the induced subgraph $G_S \coloneqq (I_m)_S$ has to be a multi-edge of type $A$, $B$ or $C$ with possibly the nodes $1$ and $2$ being white. Since $|S|=\dim\C_{I_m}+1=m+2$, in the first  two cases exactly one node has to be white (namely, node $1$ for type $A$ and node $2$ for type $B$) and in the third case both nodes have to be white. To avoid fundamental obstructions a multi-edge of type $C$ can only contain undirected and squiggly edges. This shows the following facet description for $\Tm$:

\begin{theorem}
	Let $S$ be a subset of the generators of $R_m$. Then $S$ is a facet of $\Tc_m$ if and only if $G_S$ is of one of the following three types:
	\begin{enumerate}
		\item[(i)] the node $1$ is white and the multi-edge is of type $A$.
		\item[(ii)] the node $2$ is white and the multi-edge is of type $B$.
		\item[(iii)] both nodes $1$ and $2$ are white and all edges of the multi-edge are single undirected or squiggly edges.
	\end{enumerate}
\end{theorem}
The possible graphs  $G_S$ for a facet  $S$ of $\Tc_m$ are shown in the next picture, where we draw the edges $e_1,\ldots,e_m$ from top-to-bottom:
	\begin{center}
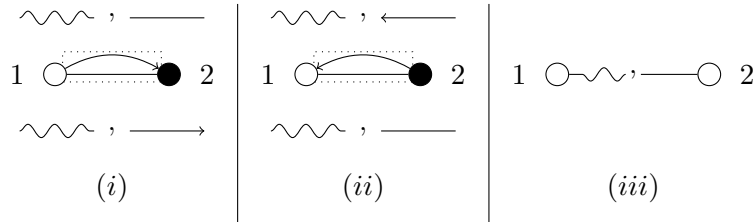
	\begin{tabular}{c|c|c}
		\begin{tikzpicture}
			\node[black, inner sep=3pt] (1) at (1.5, 0) {};
            \node[] (a) at (0.75,0.75) {,};
            \node[] (b) at (0.75,-0.75) {,};
            \node[] (2) at (-0.6,0.75) {};
            \node[] (3) at (0.65,0.75) {};
            \node[] (4) at (0.85,0.75) {};
            \node[] (5) at (2.1,0.75) {};
            \node[] (6) at (-0.6,-0.75) {};
            \node[] (7) at (0.65,-0.75) {};
            \node[] (8) at (0.85,-0.75) {};
            \node[] (9) at (2.1,-0.75) {};

            \draw[dotted] (0.1,-0.1) rectangle (1.4,0.3);
            \node[white,inner sep=3pt] (0) at (0, 0) {};
            
            \draw[squiggle] (2) to (3);
            \draw (4) to (5);
            \draw[->] (0) to [bend left] (1);
            \draw (0) to (1);
            \draw[squiggle] (6) to (7);
            \draw[->] (8) to (9);
            
            \node[] (it) at (-0.5,0) {$1$};
            \node[] (it+1) at (2,0) {$2$};
            \node (i) at (0.75,-1.5) {$(i)$};
		\end{tikzpicture}
		&
		\begin{tikzpicture}
			\node[black, inner sep=3pt] (1) at (1.5, 0) {};
            \node[] (a) at (0.75,0.75) {,};
            \node[] (b) at (0.75,-0.75) {,};
            \node[] (2) at (-0.6,0.75) {};
            \node[] (3) at (0.65,0.75) {};
            \node[] (4) at (0.85,0.75) {};
            \node[] (5) at (2.1,0.75) {};
            \node[] (6) at (-0.6,-0.75) {};
            \node[] (7) at (0.65,-0.75) {};
            \node[] (8) at (0.85,-0.75) {};
            \node[] (9) at (2.1,-0.75) {};

            \draw[dotted] (0.1,-0.1) rectangle (1.4,0.3);
            \node[white,inner sep=3pt] (0) at (0, 0) {};
            
            \draw[squiggle] (2) to (3);
            \draw[<-] (4) to (5);
            \draw[<-] (0) to [bend left] (1);
            \draw (0) to (1);
            \draw[squiggle] (6) to (7);
            \draw (8) to (9);
            
            \node[] (it) at (-0.5,0) {$1$};
            \node[] (it+1) at (2,0) {$2$};
            \node (i) at (0.75,-1.5) {$(ii)$};
		\end{tikzpicture} &
		\begin{tikzpicture}
			\node[white,inner sep=3pt] (10) at (-2, 0) {};
            \node[white,inner sep=3pt] (11) at (-4, 0) {};
            
            \draw (10) to ($(10)+(-0.9,0)$);
            \draw[squiggle] ($(10) + (-1.1,0)$) to (11);
            \node[] (15) at (-3,0) {,};

            \draw[white, draw=white] (-3,0.5) circle (2pt);

            \node[] (it) at (-4.5,0) {$1$};
            \node[] (it+1) at (-1.5,0) {$2$};
            \node (i) at (-3,-1.5) {$(iii)$};
		\end{tikzpicture}
	\end{tabular}
	\end{center}

Mimicking what we did for multicycles, we define $\widetilde{q}\in \RR^{[2]\cup E(I_m)}$ by
\[
\widetilde{q}_1=\widetilde{q}_2\coloneqq \frac{1}{2} \cdot \frac{2+\frac{1}{2}}{1+2}
\]
and
\[
\widetilde{q}_{e_i}\coloneqq \frac{1}{m}\cdot \frac{\frac{1}{2}}{1+2} \quad \mbox{ for } 1\leq i\leq m.
\]
We get the following analogue of \Cref{cor: visibility characterization}.
\begin{corollary}
\label{cor: sq db statistics trees}
    Let $S$ be a facet of $\Tm$ and let $q \in \RR^{[2]\cup E(I_m)}$ as defined above. Then the number of facets of $S$ that are visible from $q$ is equal to the number of edges in $G_S$ that are squiggly edges or double edges. 
\end{corollary}

We omit the details of the proof as it is exactly analogous to the proof of \Cref{cor: visibility characterization} when we take the hyperplanes defined therein and restrict them to the considered multi-edge. 
\Cref{thm:Ehrhart multitree} now follows from the first part of the proof of \Cref{thm: h-star of cycle with multiedges}. 

\subsection{Other graphs}
\label{subsec: other graphs}
The results in \Cref{subsec: cycles} and \Cref{subsec: tree polynomial} raise a natural question: Can the same approach be used to compute the $h^\ast$-polynomial of cosmological polytopes of other classes of graphs.
A reasonable starting point is to extend the method to families of $2$-connected graphs beyond the multicycle. 
Based on computational evidence, we propose the following conjecture for one class of $2$-connected graphs.

\begin{conjecture}
\label{conj: 2-connected}
    Let $P_k$, $P_\ell$, and $P_m$ be paths of length $k,\ell,m$, respectively, and let $G$ be the graph obtained by identifying the three starting nodes and the three ending nodes, then we have
    \begin{align*}
        h^\ast(\C_G;z)=(1+3z)^{k+\ell+m}&-(2z)^{k+\ell}(1+3z)^m-(2z)^{k+m}(1+3z)^\ell -(2z)^{\ell+m}(1+3z)^k\\&-(2z)^{k+\ell+m}+3 \cdot (2z)^{k+\ell+m}.
    \end{align*}
\end{conjecture}

Note that the family of graphs considered in \Cref{conj: 2-connected} contains the complete bipartite graph $K_{2,3}$. 
The normalized volume of $\C_{K_{2,m}}$ for $m\geq 1$ was computed by Landin in \cite{landin2023cosmological}. 
Indeed, evaluating the above polynomial at $1$ for $G = K_{2,3}$ recovers the normalized volume formula in \cite{landin2023cosmological}. 

Thinking more generally, \Cref{cor: visibility characterization} and \Cref{cor: sq db statistics trees} support the following general formula for the $h^\ast$-polynomial of a cosmological polytope.
\begin{conjecture}
    \label{conj: h*-formula}
    Let $G$ be a graph and $T$ be a triangulation of $\C_G$ given by a good term order on $R_G$. Then
    \[
    h^\ast(\C_G; z) = \sum_{S\in T}z^{\textrm{sq}(G_S) + \textrm{db}(G_S)},
    \]
    where $\textrm{sq}(G_S)$ and $\textrm{db}(G_S)$ respectively denote the number of squiggly and double edges in $G_S$.
\end{conjecture}

If \Cref{conj: h*-formula} fails to hold for an arbitrary good term order, it may be the case that, for any graph $G$, there exists (at least) one good term order yielding the combinatorial formula for $h^\ast(\C_G;z)$ specified in \Cref{conj: h*-formula}.

\subsection*{Acknowledgements}
Liam Solus was supported by the Wallenberg Autonomous Systems and Software Program (WASP) funded by the Knut and Alice Wallenberg Foundation, a Starting Grant from the Swedish Research Council (Vetenskapsr\aa{}det), the G\"oran Gustafsson Foundation Prize for Young Researchers, and a research pairs grant from the Center for Digital Futures at KTH.

\bibliographystyle{abbrvnat}
\bibliography{ehr_cosmo_bib}



\end{document}